\documentclass[a4paper,10pt,reqno, english]{amsart}

\usepackage{amsmath,amssymb,amscd,amsthm,amsfonts}
\usepackage{graphicx,subfigure}
\usepackage[pdfencoding=auto]{hyperref}
\usepackage{dsfont}
\usepackage{xcolor}
\usepackage[utf8]{inputenc}
\usepackage{microtype}

\newtheorem{theorem}{Theorem}[subsection]

\newtheorem{lemma}[theorem]{Lemma}

\newtheorem{corollary}[theorem]{Corollary}
\newtheorem{example}[theorem]{Example}
\newtheorem{conjecture}[theorem]{Conjecture}

\newtheorem{problem}[theorem]{Problem}

\newcommand{\rr}{\mathds{R}}
\newcommand{\zz}{\mathds{Z}}

\DeclareMathOperator{\conv}{conv}

\title{A survey of mass partitions}

\author[Roldán-Pensado]{Edgardo Roldán-Pensado}
\address{Centro de Ciencias Matemáticas, UNAM Campus Morelia, Morelia, Mexico}
\email{e.roldan@im.unam.mx}

\author[Soberón]{Pablo Soberón}\address{Baruch College, City University of New York, One Bernard Baruch Way, New York, NY 10010, United States} 
\email{pablo.soberon-bravo@baruch.cuny.edu}

\thanks{Roldán-Pensado's research is supported by CONACYT project 282280. Soberón's research is supported by NSF award DMS-1851420 and PSC-CUNY grant 63529-00 51.}

\begin{document}

\begin{abstract}
Mass partition problems describe the partitions we can induce on a family of measures or finite sets of points in Euclidean spaces by dividing the ambient space into pieces.  In this survey we describe recent progress in the area in addition to its connections to topology, discrete geometry, and computer science.
\end{abstract}

\maketitle

\tableofcontents

\section{Introduction}

\emph{Mass partition problems} study how partitions of Euclidean spaces split families of measures. For example, for a given measure, how can the total space be split so that each part has the same measure while satisfying some additional geometric property? These problems are also referred to as \emph{measure partition problems} or \emph{equipartition problems}.
The quintessential example is the \emph{ham sandwich theorem}. Informally, this theorem states that a (three-dimensional) sandwich made out of three ingredients may be split fairly among two guests using a single straight cut.

\begin{theorem}[Ham sandwich]
Given $d$ finite measures in $\rr^d$, each absolutely continuous with respect to the Lebesgue measure, there exists a hyperplane that divides $\rr^d$ into two half-spaces of the same size with respect to each measure.
\end{theorem}

Mass partition problems are at the crossroads of topology, discrete geometry, and computer science. These problems usually appear while studying discrete geometry and provide a natural field to test tools from equivariant topology. The explicit computation of fair partitions of finite point sets is an exciting challenge, and many such results can be applied to geometric range searching. Furthermore, mass partitions have been used successfully to solve hard problems in seemingly unrelated areas, such as incidence geometry. 

Surveys that cover mass partitions often focus on the topological \cite{Steinlein:1985wia, matousek2003using, Zivaljevic:2004vi, DeLoera:2019jb, Blagojevic:2018jc}, or computational \cite{Edelsbrunner:1987ct, Agarwal:1999vl, Matousek:2002td, Kaneko:2003jy, Kano:2020ws} aspects of this area. The purpose of this survey is to give a broad overview of mass partition theorems and recent advances in this area.

A large family of results related to mass partitions is \textit{fair partitions}, particularly \emph{cake splitting problems}. These problems often deal with partitions of an interval in $\rr$ according to several players' subjective preferences (which may or may not be induced by measures). In this survey, we focus on results where the geometry of the ambient space plays a key role. We only dwell on interval partitions when they have some interesting higher-dimensional extension.  We recommend \cite{Weller:1985kj, Brams:1996wt, Su:1999es,Barbanel:2005tc, procaccia2015cake} and the references therein for different variations of fair partitions, their applications, and methods to obtain them.

The general form of mass partition problems is the following.

\begin{problem}\label{prob:general}
Let $\mathcal{P}$ be a family of partitions of $\rr^d$, such that every partition in $\mathcal{P}$ splits $\rr^d$ into the same number of parts.  Given a family $\mathcal{H}$ of measures in $\rr^d$, is there a partition in $\mathcal{P}$ that splits each measure in a prescribed way?
\end{problem}

For example, in the ham sandwich theorem, $\mathcal{P}$ is the set of partitions of $\rr^d$ into two parts by a single hyperplane.  In most cases we seek to split each measure into parts of equal size.  If $\mathcal{P}$ is fixed, an interesting parameter is the maximal cardinality of $\mathcal{H}$ for which Problem \ref{prob:general} has an affirmative answer.   The ham sandwich theorem is optimal in this sense.  If we consider $d+1$ measures, each concentrated near one of the vertices of a simplex, no hyperplane can simultaneously halve all of them.

We consider several families $\mathcal{P}$ which lead to interesting problems.  These include partitions by several hyperplanes, partitions into convex pieces, partitions with low complexity algebraic surfaces, partitions by cones, and more.

\subsection{History}

It is hard to say precisely when the study of mass partition problems started. For example, the ham sandwich theorem in dimension one is equivalent to the existence of the median in probability.  Early results dealt with splitting the volume of convex bodies into equal parts, rather than general measures.  Most proofs can be extended to deal with measures with minor modifications.  

If we focus on results in dimension greater than one, Levi published the first mass partition result in 1930, regarding translates of cones \cite{Levi:1930ea}.

\begin{theorem}
Let $d$ be a positive integer and let $C_1, C_2, \ldots, C_{d+1}$ be $d+1$ convex cones in $\rr^d$ with apex at the origin.  Suppose the union of the cones is $\rr^d$ and that their interiors are not empty and pairwise disjoint.  Let $\mu$ be an absolutely continuous probability measure.  There exists a vector $x \in \rr^d$ such that the translates $x+C_j$ satisfy
\[
\mu(x+C_j) = \frac{1}{d+1} \qquad \text{for } j=1,\ldots, d+1.
\]
\end{theorem}

Levi also {proved} two other mass partition theorems in dimension two using continuity arguments. One of these results is the two-dimensional version of the ham sandwich theorem.  The other is that for any convex body in the plane we can always find three concurrent lines forming at angles of $\pi/3$, each of which halves the area of the convex body.

The ham sandwich theorem appeared in the 1930s.  It was conjectured by Steinhaus.  In 1938 Steinhaus presented a proof for the case $d=3$ and attributed it to Banach \cite{Steinhaus1938} (see \cite{Beyer:2018if} for a translation).
The problem itself is listed as Problem 123 of the Scottish Book \cite{Mauldin1981}, where mathematicians in Lw\'ow, Poland would gather problems, conjectures, and solutions.  The Scottish Book is named after the Scottish Caf\'e, where meetings were held.  Shortly after, Stone and Tukey proved the ham sandwich theorem in its general form \cite{Stone:1942hu}.

Steinhaus gave another, quite different, proof of the ham sandwich theorem in dimension three in 1945 \cite{Steinhaus1945}, using the Jordan curve theorem. According to \cite{Beyer:2018if}, this work was done while Steinhaus was in hiding during World War II.

In Courant and Robbins's 1941 book \cite[pp. 317]{Courant:1941ie} there are a few results regarding divisions of planar convex bodies. For example, you can split a convex body into four parts of equal areas using two straight lines. It is clear from the text that the authors were aware of the ham sandwich theorem and its relation to these results. Buck and Buck generalized this result by showing that you can split any convex body into six parts of equal area using three concurrent lines \cite{Buck:1949wa}. This is the first time that the term \emph{equipartition} was used.

These were the first results on mass partitions.  Since then, the area has been greatly developed. One of the main goals is to find natural or useful families of partitions for which an equipartition of several measures is always guaranteed to exist.

\subsection{Continuous and discrete versions}

Mass partition theorems are usually stated in one of two settings: discrete or continuous.  Figure \ref{fig:disc-cont} shows an example of each kind.  In the discrete setting, we work with finite families of points.  Some conditions are imposed to avoid degenerate partitions, i.e., those that have many points in the boundary between parts.

A set of points in $\rr^d$ is in general position if no $d+1$ of them lie on a single hyperplane.  For a finite set $P \subset \rr^d$ in general position, we say that a hyperplane $H$ \textit{halves} $P$ if either $|P|$ is even and each open side of $H$ contains exactly $|P|/2$ points of $P$, or $|P|$ is odd and each open side of $H$ contains exactly $(|P|-1)/2$ points of $P$.  Then, we can state the discrete ham sandwich theorem.

\begin{theorem}[Discrete ham sandwich]
Let $d$ be a positive integer and $P_1, \ldots, P_d$ be finite sets of points in $\rr^d$ such that $P_1 \cup \ldots \cup P_d$ is in general position.  Then there exists a hyperplane $H$ that simultaneously halves each of $P_1, \ldots, P_d$.
\end{theorem}

The continuous versions deal with measures in $\rr^d$. In older papers, the measures usually came from a density function or the volume measure of some convex region in $\rr^d$. However, most of those results can be proved for more general measures with only minor modifications to the original proofs.  Some proofs even extend to charges, which can assign negative values to subsets of $\rr^d$.

We say that a measure $\mu$ in $\rr^d$ is absolutely continuous if it is absolutely continuous with respect to the Lebesgue measure.  We say that $\mu$ is finite if $\mu(\rr^d)$ is finite.  The condition of absolute continuity is imposed to avoid worrying about the boundary between sections of a partition, as it has measure zero for most families of partitions we consider.  It also has the benefit that the functions involved in topological proofs, as described in the next subsection, are continuous.  In some cases (such as the theorem below), the absolute continuity condition can be replaced by a weaker statement, such as $\mu(H)=0$ for every hyperplane $H$.

\begin{theorem}[Continuous ham sandwich]\label{thm:continuous-hs}
Let $d$ be a positive integer and let $\mu_1, \ldots, \mu_d$ be absolutely continuous finite measures in $\rr^d$.  Then, there exists a hyperplane $H$ such that the two closed half-spaces $H^+$ and $H^{-}$ with boundary $H$ satisfy
\[
\mu_i (H^+) = \mu_i (H^-) \qquad \text{for }i=1,\ldots, d.
\]
\end{theorem}

\begin{figure}
    \centering
    \includegraphics[scale=0.8]{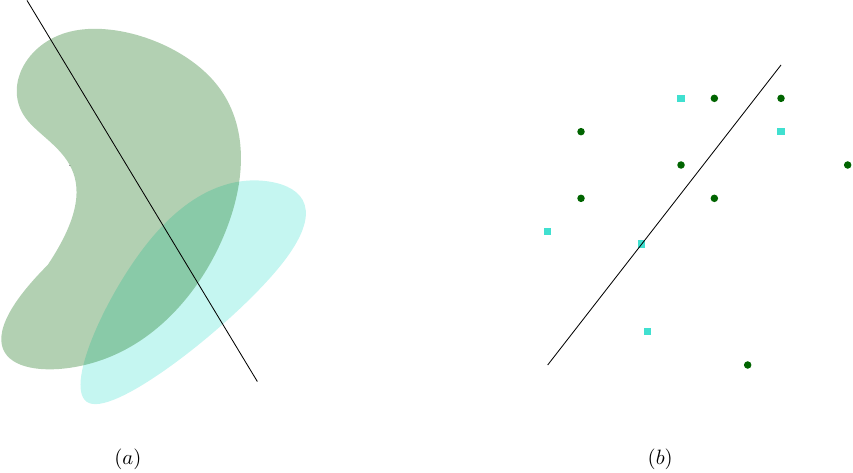}
    \caption{(a) A partition of two absolutely continuous measures in the plane. (b) A partition of two finite sets of points.  Notice the separating line only contains points from a set if it has an odd number of elements.}
    \label{fig:disc-cont}
\end{figure}

The connection between continuous and discrete versions of mass partition results can be argued by approximation arguments.  A measure can be approximated by a sequence finite set of points of increasing size.  A finite set of points $P$ can be approximated by a sequence of measures whose supports are balls centered at points of $P$, of decreasingly small radius.  The equivalence, although common, is not entirely immediate.  For some mass partition problems going from the continuous version to the discrete is a nuanced problem \cite{Blagojevic:2019gb}.  When such approximation arguments work, we immediately obtain mixed versions of mass partition results, in which some measures are concentrated in points and some are absolutely continuous.

To allow general measures or to avoid the general position assumptions, we need to modify the conclusion in our theorem \cite{Cox:1984is}.  For example, a ham sandwich theorem for a family of $d$ general measures says the following. \emph{We can always find a hyperplane $H$ such that for any measure $\mu$ the two closed half-spaces $H^+$ and $H^-$ bounded by $H$ satisfy}
\[
    \mu(H^+) \ge \frac{\mu(\rr^d)}{2}, \qquad    \mu(H^-) \ge \frac{\mu(\rr^d)}{2}.
\]

A common degenerate family of measures where this is used is formed by \emph{sets of weighted points}, which are linear combinations of Dirac measures.

\subsection{Topological proof techniques}

Mass partition problems are one of the best examples of combinatorial problems that can be solved using topological tools.  A standard way to reduce combinatorial problems to topology is via the \emph{test map/configuration space (TM/CS) scheme}.  This is a technique frequently used in topological combinatorics.  For continuous mass partition problems, it consists of the following steps.

\begin{itemize}
    \item First, we define a topological space $X$, often called the configuration space. This space parametrizes a space of partitions related to our problem.% The ultimate goal is to find a partition in $X$ that solves our problem.
    \item Second, we construct a topological space $Y$, which describes how an element of $X$ splits each measure.  There should also be a natural function $f:X \to Y$, called the test map.  The smoothness conditions on the measures are typically required to make this map continuous. It is common for $Y$ to be real space $\rr^n$ and for $f$ to be built so that any $x\in f^{-1}(0)$ is a solution to the problem.
    \item Third.  Ideally, the symmetries of the problem imply that there is a group $G$ acting on $X$ and $Y$.  This is usually chosen so that the map $f:X \to Y$ is \textit{equivariant}.  In other words, for every $x \in X$ and $g \in G$, we have
    \[
    f(gx) = g f(x).
    \]
    \item Finally, we reduce the mass partition problem to a property of $G$-equivariant continuous maps $f:X \to Y$.  Proving that all such maps satisfy the desired property is done as a purely topological problem.
\end{itemize}

Ultimately, the choice of space $X$ determines the topological problem in the last step.  Let us prove the ham sandwich theorem to exemplify this method.  The topological tool we use is the Borsuk--Ulam theorem.

\begin{theorem}[Borsuk--Ulam]
Let $S^{d} \subset \rr^{d+1}$ be the $d$-dimensional sphere.  For any continuous map $f: S^d \to \rr^d$ such that $f(x) = -f(-x)$ for all $x \in S^d$, there exists $x \in S^d$ such that $f(x) = 0$.
\end{theorem}

\begin{proof}[Proof of Theorem \ref{thm:continuous-hs}]
Let us parametrize the space $X$ of all closed half-spaces in $\rr^d$.  We can assign to each closed half-space $H^+$ uniquely a pair $(v,\alpha) \in S^{d-1} \times \rr$ so that 
\[
H^+ = \{x \in \rr^d : \langle x, v \rangle \le \alpha \},
\]
where $\langle \cdot, \cdot \rangle$ denotes the standard dot product.  Note that $(v,\alpha)$ and $(-v,-\alpha)$ correspond to complementary half-spaces that share the bounding hyperplane $\{x \in \rr^d : \langle v, x\rangle = \alpha\}$.  Let $e_{d+1}$ be the last element of the canonical basis of $\rr^{d+1}$.
Finally, the space $S^{d-1} \times \rr$ is homeomorphic to $S^{d}\setminus\{\pm e_{d+1}\}$ by mapping
\begin{align*}
    S^{d-1} \times \rr & \to S^{d}\setminus\{\pm e_{d+1}\} \\
    (v, \alpha) & \mapsto \frac{1}{\sqrt{1 + \alpha^2}}(v,\alpha).
\end{align*}
Therefore, we can use the inverses of the maps above to assign to each $x \in S^{d}\setminus\{\pm e_{d+1}\}$ a closed half-space $H(x) \subset \rr^d$.  We define $H(e_{d+1}) = \rr^d$ and $H(-e_{d+1}) = \emptyset$.  If $\mu$ is any absolutely continuous finite measure in $\rr^d$, this makes $\mu \circ H: S^d \to \rr$ a continuous map.  For our $d$ measures $\mu_1, \dots, \mu_d$, we can now consider the continuous map
\begin{align*}
    f:S^d & \to \rr^d \\
    x & \mapsto \Big(\mu_1 (H(x)) - \mu_1 (H(-x)), \ldots, \mu_d(H(x)) - \mu_d(H(-x))\Big).
\end{align*}

Notice that $f(x) = -f(-x)$, so this map must have a zero.  If $f(x) = 0$, it is clear that $x \neq \pm e_{d+1}$.  Since $H(x)$ and $H(-x)$ make complementary half-spaces, their common bounding hyperplane splits each measure by half, as we wanted.
\end{proof}

\begin{figure}
    \centering
    \includegraphics{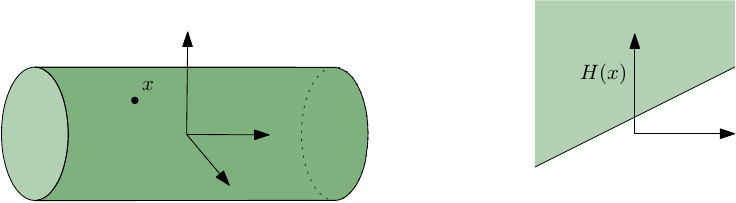}
    \caption{An infinite two-dimensional cylinder in $\rr^3$ parametrizes half-planes in $\rr^2$.  Moving across a circular cut of the cylinder corresponds to a rotation.  Moving along the cylinder corresponds to a translation.  A two-point compactification of the cylinder into a sphere provides the configuration space we seek.}
    \label{fig:cylinder1}
\end{figure}

In the proof above, the space of partitions was parametrized by $X=S^d$.  The group acting on $S^d$ is $\mathbb{Z}_2$ with the antipodal action.  The parametrization makes $x$ and $-x$ correspond to complementary half-spaces.  Different mass partition results lead us to use more elaborate configuration spaces and group actions \cite{Zivaljevic:2004vi, Blagojevic:2018jc, matousek2003using}.  The use of the Borsuk--Ulam theorem in this proof is mostly unavoidable, as the ham sandwich theorem is equivalent to the Borsuk--Ulam theorem \cite{cs_et_al:LIPIcs:2017:7232}.  The fact that a sphere parametrizes the set of half-spaces will be used repeatedly throughout this survey.

A common generalization of the Borsuk-Ulam theorem used in this proof scheme is the following theorem.  We say that the action of a group $G$ on a topological space $X$ is free if the equation $gx = x$ implies that $g$ is the neutral element of $G$. 

\begin{theorem}[Dold 1983 \cite{Dold:1983wr}]\label{thm:dold}
Let $G$ be a finite group, $|G|>1$, $X$ be an $n$-connected spaces with a free action of $G$, and $Y$ be a paracompact topological pace of dimension at most $n$ with a free action of $G$.  Then, there exists no $G$-equivariant continuous map $f:X \to G$.
\end{theorem}

The advantage of the theorem above is that we only need to know the connectedness of $X$ and the dimension of $Y$ to apply it.

\subsection{Computational complexity}

Topological proofs provide a clear and elegant way to tackle mass partition problems.  However, these are all existence proofs, giving us little information about how to find such partitions.

In the discrete versions of mass partitions results, we want to split several finite families of points in $\rr^d$ in a predetermined manner.  If the total number of points is $n$, an algorithm to find such a partition with running time in terms of $n$ is desirable.  Moreover, since mass partitioning results have applications in computer science, finding such algorithms is more than an academic exercise.  A fundamental case is the problem of finding ham sandwich partitions.

\begin{problem}
Suppose we are given $d$ families of points $P_1, \ldots, P_d$ in $\rr^d$, each with an even number of points and such that $P_1 \cup P_2 \cup \ldots \cup P_d$ is in general position.  Design an algorithm that finds a hyperplane that simultaneously halves all families of points.  The running time should be given in terms of $n = |P_1| + \ldots +|P_d|$.
\end{problem}

The search space for such an algorithm has size $O(n^{d+1})$, as that is the number of different subsets of $P_1 \cup \ldots \cup P_d$ an affine hyperplane can cut off.

In the plane, Lo, Matoušek, and Steiger \cite{Lo:1994gq} presented algorithms that run in $O(n)$ time, which is optimal.  If the points are weighted, then a ham sandwich cut can be computed in $O(n \log n)$ time \cite{Bose:2005gn}.  In dimension three, algorithms that run in $O(n^{3/2})$ time, up to poly-logarithmic factors, are known \cite{Lo:1994gq}.

In high dimensions, the current best algorithms run in $O(n^{d-1})$ time \cite{Lo:1994gq}.  For comparison, the problem requires $\Omega(n^{d-2})$ time \cite{Lo:1994gq}.  Geometric conditions can sharply reduce the computational complexity.  If we impose a separation condition on the convex hulls of each of $P_1, \ldots, P_n$ as considered by Bárány, Hubard, and Jerónimo \cite{Barany:2008vv}, then a ham sandwich cut can be found in $O(n)$ time, regardless of the dimension \cite{Bereg:2012ib} (more general cuts can be found in $O(n \log n)$ time \cite{Steiger:2010bp}).  In dimension two, some algorithms dynamically maintain ham sandwich cuts for two sets of points subject to successive insertion or deletion of points \cite{Abbott:2009do}.  The discrete ham sandwich problem, where the dimension is part of the input, is PPA-complete \cite{Goldberg:2019vv}.  This complexity class (short for ``Polynomial Parity Argument'') was introduced by Papadimitriou \cite{Papadimitriou:1994wp} and is related to the problem of finding a second vertex of odd degree in a graph where one such vertex is known to exist.

In general, the algorithmic complexity of the results in this survey is much better understood in dimensions two and three.  Every mass partition problem has an associated algorithmic variant, which is worth pursuing.

\subsection{Applications}

One of the principal applications of mass partition results is to give structure to geometric data.  This is often the case to deal with geometric range queries.  Suppose we are given a finite family $P$ of points in $\rr^d$, which is fixed.  Then, we will be given a set $C \subset \rr^d$ and might be interested in either
\begin{itemize}
    \item how many points of $P$ are contained in $C$ or
    \item a list of the points of $P$ contained in $C$.
\end{itemize}
If we want to solve this problem for a particular set $C$, we need to check one by one the points of $P$ to see if they are contained in $C$.  However, if we intend to solve this problem for a sequence of sets, called \textit{ranges}, $C_1, C_2, \ldots$, it is possible we find improved algorithms if the sets $C_i$ satisfy interesting geometric conditions.  Such conditions include being a half-space, a simplex, or an axis-parallel box.

We can use mass partition results to \textit{pre-process} the set $P$ and allow us to solve these problems efficiently.  The first use of partition results of this kind goes back to Willard \cite{Willard:1982cy} using \textit{partition trees} for simplex range searching.  The idea is to apply a mass partition result to $P$ and split it into sets $P_1, \ldots, P_k$.  Then, the mass partition result is applied to each of the $P_i$.  We continue to do so and obtain a $k$-ary tree structure on $P$.  Depending on the geometric properties of the ranges, such a tree structure can allow us to answer ranges queries efficiently.   Many of the results discussed in Section \ref{subsec:transversals} were designed to solve half-space and simplex range queries this way, such as the Yao--Yao theorem \cite{Yao:1985hf}.  Queries in which the ranges are half-spaces are relevant in database searching.

Approximate partitions are often sufficient for these problems, but easier to compute \cite{Haussler:1987fr, Matousek:1992ed}.  Other partition results, such as the cutting lemma (Theorem \ref{theorem-cutting-lemma}), also have strong applications in computational geometry.  We recommend Agarwal and Erickson's survey on geometric range queries for more on this topic \cite{Agarwal:1999vl}.

The ham sandwich theorem also has applications in voting theory \cite{Cox:1984is}, and some of its extensions can be applied to congressional district drawing \cite{Humphreys:2011iy, Soberon:2017kt}.  Section \ref{subsection:necklace} contains some examples of purely combinatorial problems in which mass partition problems are relevant.  In the next sections, we discuss the applications to other problems in discrete geometry, such as incidence geometry (Section \ref{subsec:transversals}) and geometric Ramsey theory (Section \ref{subsec:same-type}).

\section{Partitions by multiple hyperplanes}

In the ham sandwich theorem, the partitions we seek are given by a single hyperplane.  As hyperplanes are easy to parametrize, it is convenient to look at partitions induced by several hyperplanes.  The combinatorics of the complement of hyperplane arrangements has been extensively studied \cite{zaslavsky1975facing, orlik2013arrangements}.  They provide rich configuration spaces for mass partition problems.

\subsection{The Grünbaum--Hadwiger--Ramos problem}

A family of hyperplanes in $\rr^d$ is in general position if no $d+1$ of them are concurrent, and any $d$ of their normal vectors are linearly independent. If we are given $k$ affine hyperplanes in $\rr^d$ in general position and $k \le d$, then they split $\rr^d$ into $2^k$ regions.  We are interested in partitions of this type where each of the $2^k$ parts has the same size for many measures, simultaneously.

Grünbaum asked if it is possible to find such a partition for a single measure and $d$ hyperplanes \cite{Grunbaum:1960ul}.  This is simple for $d=1,2$ and known for $d=3$ \cite{Hadwiger1966,Yao:1989ha}.  For $d=2,3$ there is a continuum of equipartitions of this type, and additional conditions may be imposed.
Avis \cite{Avis:1984is} showed that we cannot guarantee the existence of $d$ hyperplanes that split a single measure into $2^d$ equal parts for $d \ge 5$. It suffices to consider a measure concentrated around points on the moment curve $\gamma(t) = (t,t^2, \dots, t^d)$ in $\rr^d$; a family of $d$ hyperplanes intersects the moment curve in at most $d^2$ points, which less than the $2^d-1$ cuts needed to guarantee the desired partition.
Because of this, Ramos proposed the following problem.

\begin{problem}\label{problem-gunbaum-hadwiger-ramos}
Determine the triples $(d,k,m)$ of positive integers such that the following statement holds.  For any $m$ absolutely continuous finite measures in $\rr^d$, there exist $k$ affine hyperplanes dividing $\rr^d$ into $2^k$ parts of equal size in each of the $m$ measures.	
\end{problem}

This is now known as the Grünbaum--Hadwiger--Ramos mass partition problem. Ramos extended Avis' argument to show that the condition $d \ge \left(\frac{2^k-1}{k}\right)m$ is necessary, and made the following conjecture \cite{Ramos:1996dm}.

\begin{conjecture}
	Let $d,k,m$ be positive integers.  The triple $(d,k,m)$ is a solution for Problem \ref{problem-gunbaum-hadwiger-ramos} if and only if
	\[
	d \ge \left\lceil  \left( \frac{2^k-1}{k}\right)m\right\rceil.
	\]
\end{conjecture}
  
The best general upper bound for this problem is by Mani-Levitska, Vrećica, and Živaljević \cite{ManiLevitska:2006vl}.

\begin{theorem}
	Let $d,k,m,a$ be positive integers such that $2^a \le m < 2^{a+1}$.  The triple $(d,k,m)$ is a solution for Problem \ref{problem-gunbaum-hadwiger-ramos} if 
	\[
	d \ge  m + \left(2^{k-1}-1 \right)2^a.
	\]
\end{theorem}

Grünbaum's problem is settled in all dimensions except $d=4$, which leaves the following question open.

\begin{problem}
Let $\mu$ be a finite absolutely continuous measure in $\rr^4$.  Decide if there always exists four hyperplanes that divide $\rr^4$ into $16$ parts of equal $\mu$-measure.
\end{problem}

The natural configuration space of a set of $k$ hyperplanes in $\rr^d$ is $\left( S^d\right)^k$, a $k$-fold direct product of $d$-dimensional spheres.  This space does not have the connectedness needed for simple topological approaches to yield strong results; for instance, we cannot obtain interesting results by applying Theorem \ref{thm:dold}.  The subtleties of the topological tools needed, as well as an excellent description of different ways to approach this problem, is best explained in the recent survey of Blagojević, Frick, Haase, and Ziegler \cite{Blagojevic:2018jc}.  Further advances can also be found in \cite{Blagojevic:2016ud}.  Vrećica and Živaljević proposed a different approach to fix some issues raised by the survey from Blagojević et al. in \cite{Vrecica:2015uj}. 

The algorithmic aspect of Problem \ref{problem-gunbaum-hadwiger-ramos} is interesting.  We can split a set of $n$ points in the plane into four parts of equal size using two lines, and the two lines can be found in $O(n)$ time \cite{Megido:1985tn}.  If we want the two lines to be orthogonal, they can be found in $\Theta (n \log n)$ time \cite{Roy:2007tc}.  The orthogonality condition yields another variant of this problem in high dimensions, where the $k$ hyperplanes must be pairwise orthogonal.

The best bounds for the orthogonal case were proved by Simon \cite{Simon:2019gj}.  If we relax the conditions on the partition, we can obtain sharp results.  Makeev proved that for any absolutely continuous measure $\mu$ in $\rr^d$ we can find $d$ pairwise orthogonal hyperplanes such that any two of them split $\mu$ into four equal parts \cite{Makeev:2007dx}.
Additionally, the orthogonality condition may be dropped to extend the equipartition to two centrally symmetric measures. To be precise, for any two absolutely continuous measures $\mu_1, \mu_2$ in $\rr^d$ which are centrally symmetric around the origin there are $d$ hyperplanes such that any two of them split both $\mu_1$ and $\mu_2$ into four equal parts.

\subsection{Successive hyperplane partitions}

A different variation of the Grünbaum--Hadwiger--Ramos problem appears if we don't allow the hyperplanes to extend indefinitely.  We say a partition $(C_1,\dots, C_{n})$ is a \textit{successive hyperplane partition} of $\rr^d$ if it can be constructed in the following way. First, use a hyperplane to split $\rr^d$ into two parts.  Then, if $\rr^d$ is split into $k$ parts for some $k<n$, use a hyperplane to cut exactly one of the parts into two.  A successive hyperplane partition of $\rr^d$ into $n$ parts always uses $n-1$ hyperplanes.

It is easy to see that $n$ must be even to split simultaneously two absolutely measures in $\rr^d$ into $n$ equal parts using a successive hyperplane partition of $\rr^d$.  For an odd number of parts, consider two uniform measures on concentric spheres of different radii.  If there was such an equipartition, the first hyperplane must leave a $k/n$ fraction of each measure on one side and an $(n-k)/n$ fraction on the other, for some integer $1\le k \le n-1$, which is not possible.  It is not clear if the parity of $n$ is the only obstacle.

\begin{problem}
	Let $n$ and $d$ be positive integers and $\mu_1, \dots, \mu_d$ be absolutely continuous measures in $\rr^d$.  Determine if it is always possible to find a successive hyperplane partition of $\rr^d$ into $2n$ parts that have the same size in each measure. 
\end{problem}

For $d=2$, the case when $\mu_1, \mu_2$ are respectively uniformly distributed in two convex sets $A, B$ with $A \subset B$ was solved affirmatively by Fruchard and Magazinov \cite{Fruchard:2016bv}.

Consider a sphere in $\rr^d$ centered at the origin.  If we construct a successive hyperplane partition whose hyperplanes all contain the origin, we obtain a partition of the sphere by great circles.  Such partitions were used by Gromov to find convex partitions of spheres \cite{Gromov:2003ga} with additional constraints.  Other results involving successive hyperplane partitions are described after Problem \ref{prob:partconvex}.
%\textcolor{red}{Creo que esto se debe mencionar aquí.}

\subsection{Bisections by hyperplane arrangements}

An alternate distribution induced by affine hyperplane arrangements is to split $\rr^d$ into two sets by a chessboard coloring, as shown in Figure \ref{fig:chessboard}.  Given a finite family $\mathcal{F}$ of affine hyperplanes, we can choose arbitrarily for each $H \in \mathcal{F}$ its positive half-space $H^+$ its negative half-space $H^-$.  This generates a partition of $\rr^d$ into two parts $A,B$ defined by
\begin{align*}
	A & = \{x \in \rr^d : x \in H^+ \text{ for an even number of hyperplanes } H \in \mathcal{F}\}, \\
	B & = \{x \in \rr^d : x \in H^+ \text{ for an odd number of hyperplanes } H \in \mathcal{F}\}.
\end{align*}

\begin{figure}[ht]
    \centering
    \includegraphics{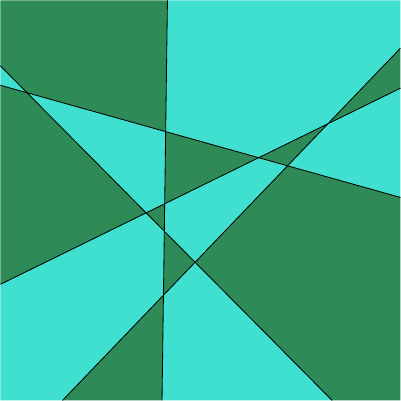}
    \caption{Example of induced chessboard coloring for five lines in the plane.}
    \label{fig:chessboard}
\end{figure}

High-dimensional results of this kind were proved by Alon and West \cite{Alon:1985cy}.  The partitioning hyperplanes were restricted: only hyperplanes orthogonal to the elements of the canonical basis were allowed, and the partition must remain invariant after reorderings of the canonical base.  They showed that, if $d$ is an odd positive integer, it is possible to split any $m$ measures using $m$ cuts for each coordinate axis.

A less restrictive result is obtained if we only fix the direction of each hyperplane. If we declare that there must be $n_i$ hyperplanes orthogonal to the $i$-th element of the canonical basis, then Karasev, Roldán-Pensado, and Soberón showed that $n_1+\dots+n_d$ measures may be split whenever the multinomial coefficient
\[
    \binom{n_1+\dots+n_d}{n_1,\dots,n_d}
\]
is odd \cite{Karasev:2016cn}. This last condition is equivalent to $n_i$ and $n_j$ not sharing a $1$ in the same position of their binary expansions.

If we completely remove the conditions on the hyperplanes, one would expect to be able to partition more measures.  The following conjecture by Langerman was presented by Barba, Pilz and Schnider \cite{Barba:2019to} when they solved the case $d=n=2$ positively.

\begin{conjecture}\label{conjecture-langerman}
	Let $n,d$ be positive integers.  For any family of $nd$ absolutely continuous measures in $\rr^d$ there exists a set of $n$ hyperplanes such that their induced chessboard coloring splits each measure into two equal parts.
\end{conjecture}

This conjecture has been verified when $n$ is a power of two by Hubard and Karasev \cite{Hubard:2019we}.  If $n$ is odd and $d-\ell$ is a power of two, it was shown that $(d-\ell)n + \ell$ measures can be simultaneously split into two equal parts by $n$ hyperplanes \cite{Blagojevic:2018ve}.  Schnider proved a different relaxation of Conjecture \ref{conjecture-langerman}, namely, he proved the following result \cite{Schnider:2020kk}.

\begin{theorem}
	Given $nd$ absolutely continuous finite measures in $\rr^d$, there exists a family $\mathcal{H}$ of $n$ affine hyperplanes such that the following statement holds.  For each measure $\mu$ in the family, either $\mathcal{H}$ halves $\mu$, or there exists $H \in \mathcal{H}$ such that $\mathcal{H} \setminus \{H\}$ halves $\mu$.
\end{theorem}

We include a simple proof of the following weaker form of Conjecture \ref{conjecture-langerman} for $n=2$. This result will be used in Section \ref{subsection-lines} and showcases the methods and obstacles involved in solving Conjecture \ref{conjecture-langerman}.  The methods of \cite{Hubard:2019we, Blagojevic:2018ve} imply Theorem \ref{theorem-weak} with $2^{a+1}$ measures instead of $2^{a+1}-1$ measures if $d=2^a$.  We discuss this improvement after the proof.  For large values of $d$, the number of measures can be increased to $2d - O(\log d)$ \cite{Blagojevic:2018ve}.

\begin{theorem}\label{theorem-weak}
	Let $a, d$ be positive integers such that $2^{a+1}> d \ge 2^a$.  For any $2^{a+1}-1$ absolutely continuous measures in $\rr^d$ there exist two hyperplanes whose induced chessboard coloring splits each measure into two equal parts.
\end{theorem} 

\begin{proof}
	The space of closed half-spaces in $\rr^d$ can be parametrized by $S^d$.  For $x \in S^d$, we denote by $H(x)$ associated half-space.  The space of pairs of half-spaces is therefore $S^d \times S^d$.  For $(x,y) \in S^d \times S^d$ we consider the two sets
	\begin{align*}
		A &= \Big{[} H(x) \cap H(y) \Big{]} \cup \Big{[} H(-x) \cap H(-y) \Big{]}, \\
		B &= \Big{[} H(-x) \cap H(y) \Big{]} \cup \Big{[} H(x) \cap H(-y) \Big{]}.
	\end{align*}
	
	Let $m=2^{a+1}-1$ and $\mu_1, \dots, \mu_{m}$ be the measures we want to split.  Consider the function
	\begin{align*}
		f: S^d  \times  S^d &\to \rr^m \\
		(x,y) & \mapsto  (\mu_1 (A) - \mu_1(B), \dots, \mu_m(A) - \mu_m(B)).
	\end{align*}
	
	The topological tool we use is the following consequence of Fadell and Husseini \cite{Fadell:1988tm}, which has been simplified by Ramos \cite{Ramos:1996dm} and Chan et al. \cite{Chan:2020vd}.  For integers $d, l_1, l_2$ such that $0 \le l_1 \le d$ and $0 \le l_2 \le d$ consider a continuous map $g: S^d \times S^d \to \rr^{2d-l_1-l_2} \times \rr^{l_1} \times \rr^{l_2}$ such that 
	\begin{align*}
		\text{ if } \quad (x,y) & \mapsto (u,v,w), \quad\text{ then } \\
		(-x,y) & \mapsto (-u,-v,w), \\
		(x,-y) & \mapsto (-u,v,-w).
	\end{align*}
	This map has a zero as long if $\binom{2d-l_1-l_2}{d-l_1}$ is odd.  This last condition holds if and only if $d-l_1$ and $d-l_2$ share no ones in the same position in their expansion base two.
	
	If we consider $l_1 = 1+2d-2^{a+1}$, $l_2 = 0$ we have the parity condition required.  We can define 
	\[
	g(x,y) = (f(x,y),0,0) \in \rr^{2d-l_1-l_2} \times \rr^{l_1} \times \rr^{l_2}.
	\]
	A zero of this function implies the existence of a zero of $f$, as we wanted.
\end{proof}

The reason we lost the ability to split one measure fewer than \cite{Hubard:2019we, Blagojevic:2018ve} is that we are not using all the properties of $f$.  In our construction, notice $f(x,y) = f(y,x)$, yet this is not used in the proof above.  Hubard and Karasev prove that adding this to the group action changes the topological obstruction and allows us to split one more measure. 

\subsection{Necklace splitting}\label{subsection:necklace}  The name for this family of partition results comes from the following setting.  Suppose $r$ thieves steal an unclasped necklace with $m$ types of beads.  The number of beads of each kind is a multiple of $r$.  They want to cut the necklace into several strings and distribute the strings among themselves so that each thief has the same number of beads of each kind.  What is the minimum number of cuts needed to obtain such a partition?  Figure \ref{fig:necklace}(b) shows how to construct examples that require $(r-1)m$ cuts, by placing the beads in $m$ monochromatic intervals.  Each interval requires at least $r-1$ cuts.  In 1987, Alon proved that $(r-1)m$ cuts were always sufficient \cite{Alon:1987ur}.  He proved both the continuous and discrete versions of this result.  In the continuous version, the set of beads of a given color is replaced by a finite absolutely continuous measure on $\rr$.

\begin{figure}[h]
    \centering
    \includegraphics{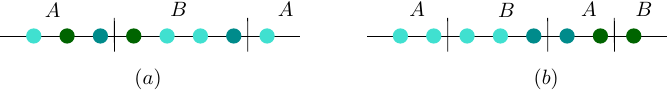}
    \caption{Two necklaces split with three kinds of beads divided fairly among two thieves.}
    \label{fig:necklace}
\end{figure}

\begin{theorem}[Necklace splitting theorem]\label{theorem-alon-necklaces}
	Let $m, r$ be positive integers and let $\mu_1, \dots, \mu_m$ be $m$ absolutely continuous probability measures in $\rr$. There exists a partition of $\rr$ using $(r-1)m$ cuts such that the resulting $(r-1)m+1$ intervals can be distributed among $r$ sets $A_1, \dots, A_r$ such that
	\[
	\mu_i (A_j) = \frac{1}{m} \qquad \text{for } i=1,\dots, m,\quad j=1,\dots,r.
	\]
\end{theorem}

Hobby and Rice first proved the continuous version for $r=2$ \cite{Hobby:1965bh}.  Goldberg and West then proved the discrete version for $r=2$ \cite{Goldberg:1985jr}, and a second proof was presented by Alon and West \cite{Alon:1985cy}.  The discrete version is a completely combinatorial problem: the number of beads of each kind and the order in which they come determine the problem.  Nevertheless, the proofs for the vast majority of cases are essentially topological (say, using discrete analogues of standard topological machinery \cite{Palvolgyi:2009vy, Meunier:2014dz}).  The cases for $r=2$ and any $m$, and for $m=2$ and any $r$ admit inductive proofs \cite{Meunier:2008jh, Asada:2018ix}.

The computational complexity for the discrete necklace problem has been recently settled.  In the case $r=2$, it is ``Polynomial Parity Argument'' (PPA) complete \cite{Goldberg:2019vv}.  This complexity class often appears with problems related to the Borsuk-Ulam theorem \cite{Papadimitriou:1994wp}.  For $r>2$, the discrete necklace problem is related to the $PPA_q$ complexity classes, which extend PPA from parity arguments to modulo $q$ arguments \cite{hollender2019classes, filos2020topological, goos2019complexity}.  Efficient approximation algorithms have been found if the $r$ thieves are satisfied with a smaller portion of each kind of bead, yet still proportional to $1/r$ \cite{Alon:2020wq}.

We present two proofs of Theorem \ref{theorem-alon-necklaces} for $r=2$ to illustrate the main ideas behind this family of results.

\begin{proof}[First proof of the Hobby--Rice theorem]
	Map the necklace to the moment curve $\gamma(t) = (t,t^2, \dots, t^m)$ in $\rr^m$.  We now have $m$ measures on a curve in $\rr^m$, so we can apply the ham sandwich theorem to them.  A hyperplane cuts the moment curve in at most $m$ points.  A thief receives all the intervals on one side of the hyperplane, and the rest go to the other thief.
\end{proof}

\begin{proof}[Second proof of the Hobby--Rice theorem]
	The necklace can be identified with the $[0,1]$ interval.  For a partition of the necklace into $m+1$ intervals, let $x_i$ be the length of the $i$-th interval.  If we distribute the intervals among two thieves $A, B$ let
	\[
	\lambda_i = \begin{cases}
		1 & \text{if the $i$-th interval was given to $A$}, \\
		-1 & \text{if the $i$-th interval was given to $B$.}
	\end{cases}
	\]
	The vector $(\lambda_1 x_1, \dots, \lambda_{m+1}x_{m+1})$ is a point on the boundary of the unit octahedron in $\rr^{m+1}$
	\[
	\mathbb{O}^m = \{(y_1, \dots, y_{m+1}) \in \rr^{m+1}: |y_1|+ \dots + |y_{m+1}| = 1\} \cong S^m.
	\]  Moreover, the antipodal action on $\mathbb{O}^m$ corresponds to flipping the assignment of intervals between the thieves $A$ and $B$.  If $\mu_1, \dots, \mu_m$ are the measures on $[0,1]$, we consider the map
	\begin{align*}
		f:\mathbb{O}^m & \to \rr^m \\
		(\lambda_1 x_1, \dots, \lambda_{m+1} x_{m+1}) & \mapsto (\mu_1 (A) - \mu_1(B), \dots, \mu_m(A) - \mu_m(B)).
	\end{align*}
	We can verify that the map is continuous and odd.  By the Borsuk-Ulam theorem, it has a zero, finishing the proof.
\end{proof}

In the first proof of the Hobby--Rice theorem, Asada et al. observed that we may use other results regarding partition by hyperplane arrangements (such as those from the Grünbaum-Hadwiger-Ramos problem) instead of the ham sandwich theorem.  We can therefore impose additional conditions on the necklace splittings \cite{Asada:2018ix, Jojic:2020uv}.

Even though there are examples of necklaces that require $(r-1)m$ cuts, sometimes fewer cuts are sufficient. For the case $r=2$, the case of partitions by $m-1$ cuts yields interesting results \cite{Simonyi:2008wm}, and partitions by $m-3$ cuts may be very hard to obtain \cite{Alon:2009gf, Lason:2015hj}.  Deciding if a necklace with $m$ kinds of pearls can be distributed among two thieves using fewer than $m$ cuts is an NP-complete problem \cite{Meunier:2008jh}.

\begin{problem}
Characterize the necklaces with $m$ types of beads, and a multiple of $r$ of each type of bead that require exactly $(r-1)m$ cuts to be fairly distributed among $r$ thieves.
\end{problem}
 
There are several high-dimensional extensions of the necklace splitting theorem, depending on which partitions we consider for $\rr^d$.  If we are given $m$ measures in the unit cube $[0,1]^d$, it was proved by de Longueville and Živaljevic that it is possible to distribute them among $r$ thieves by using $(r-1)m$ cuts by hyperplanes parallel to the facets of the hypercube and then distributing the pieces, even if the number of cuts parallel to each facet is fixed in advance \cite{deLongueville:2006uo}.

\begin{figure}[ht]
    \centering
    \includegraphics[scale=0.7]{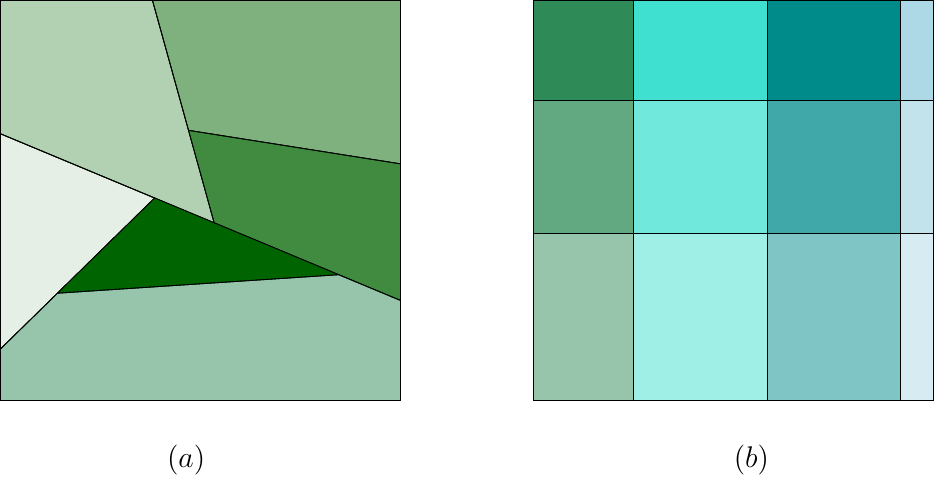}
    \caption{(a) A partition of successive hyperplane partition $\rr^2$ into seven pieces using six lines.  For necklace splitting, this reduces the number of parts to distribute. (b) A partition as used by de Longueville and Živaljević \cite{deLongueville:2006uo}, this optimizes the number of cutting hyperplanes. }
    \label{fig:dim2necklace}
\end{figure}

Instead of optimizing the number of hyperplanes used to make the partition, we can reduce the number of parts into which we split $\rr^d$, if we restrict attention to convex pieces.

\begin{problem}\label{prob:partconvex}
Let $r, m, d$ be positive integers.  Find the smallest value $k=k(r,m,d)$ such that the following statement holds.  For any $m$ absolutely continuous probability measures $\mu_1, \dots, \mu_m$ in $\rr^d$, there exists a partition of $\rr^d$ into $k$ convex parts that can be distributed among $r$ sets $A_1, \dots, A_r$ such that	
\[
\mu_i (A_j) = \frac{1}{r} \qquad \text{for } i=1,\dots, m, \quad j=1,\dots, r.
\]
\end{problem}

Theorem \ref{theorem-alon-necklaces} implies $k(r,m,1) = (r-1)m+1$.  If the partitions are made by successive hyperplanes with directions fixed in advance, and each hyperplane only cuts one part, Karasev, Roldán-Pensado, and Soberón showed that $k(r,m,d) \le (r-1)m+1$ \cite{Karasev:2016cn}.  Blagojević and Soberón proved that $k(r,m,d) \le \frac{m(r-1)}{d}$, so the thieves can use the dimension to their advantage \cite{Blagojevic:2018gt}.

The discrete necklace splitting theorem is a surprising application of the ham sandwich theorem to a completely combinatorial problem.  The following extension of a problem for the All-Russian Mathematical Olympiad 2005 provides another example \cite{BKC05}.

\begin{example}\label{example-russian}
We are given $n$ baskets, each containing a finite (possibly zero) amount of $q$ different types of fruits.  Any basket may have a positive amount of more than one fruit.  Prove that it is possible to choose no more than $(n+q)/2$ baskets and have at least half of the total amount of each kind of fruit.
\end{example}

\begin{proof}[Solution.]
    Let $\varepsilon >0$.  For each basket, choose a ball of radius $\varepsilon$ in $\rr^q$.  We choose the centers of the balls such that no hyperplane intersects more than $q$ of the balls.  We represent the fruits in a basket by weighted points in the corresponding ball.  Now, since we have $k$ weighted sets of points in $\rr^q$, we can have a hyperplane that simultaneously halves all of them.  Suppose the hyperplane intersects $r$ of the balls.  One of the two open half-spaces determined by this hyperplane contains at most $(n-r)/2$ of the remaining balls.  By choosing the baskets on that side of the hyperplane and those intersecting the hyperplane, we are guaranteed to have at least half of each kind of fruit.  Moreover, the number of baskets kept is $(n+r)/2 \le (n+q)/2$.
\end{proof}

If each type of fruit is distributed evenly among an odd number of baskets and no basket contains more than one kind of fruit, we can see that the solution above yields an optimal bound.  The original problem was the case $n=101$, $q=3$, which can be solved by purely combinatorial methods as well.

\section{Convex Partitions of \texorpdfstring{$\rr^d$}{Rd}} \label{section-convexpartitions}

If we seek to split $\rr^d$ into more than one piece, we can ask for the parts to be convex.  We say that $(C_1, \dots, C_k)$ is a \textit{convex partition} of $\rr^d$ into $k$ parts if 
\begin{itemize}
	\item every set $C_i$ is a closed and convex subset of $\rr^d$,
	\item the union of all $C_i$ equals $\rr^d$, and
	\item the interiors of any two $C_i, C_j$ are disjoint if $i \neq j$. 
\end{itemize}

Many natural partitions of $\rr^d$, such as partitions induced by hyperplane arrangements, are convex partitions.  Yet, the space of convex partitions of $\rr^d$ into $k$ parts is hard to parametrize \cite{Leon:2018ex}.  Several proofs concerning convex partitions of $\rr^d$ instead use a subset of those partitions that are easier to parametrize: \textit{generalized Voronoi diagrams}, also called \textit{power diagrams}.

Given a family of $k$ different points $x_1, \dots, x_k$ in $\rr^d$, denoted \textit{sites}, and $k$ real numbers $\alpha_1, \dots, \alpha_k$, we can define the $k$ functions
\begin{align*}
	f_j : \rr^d & \to \rr \\
	x & \mapsto \operatorname{dist}(x,x_j)^2 - \alpha_j.
\end{align*}

Then, we consider the sets $C_j = \{x \in \rr^d : f_j(x) \le f_i (x) \text{ for } i = 1,\dots, k\}$.  It is a simple exercise to show that these sets form a convex partition of $\rr^d$.  If $\alpha_1 = \dots = \alpha_k$, we have a Voronoi diagram.  If we fix the points $x_j$ and an absolutely continuous finite measure $\mu$, we can find values $\alpha_1, \dots, \alpha_k$ such that the values $\mu(C_j)$ match any numbers we want, provided they sum to $\mu(\rr^d)$ \cite{Aurenhammer:1998tj}.  We use this result again in Section \ref{subsec:fans-cones}.  The space of possible $k$-tuples of different points is the standard \textit{configuration space} of $\rr^d$, which is widely used in algebraic topology \cite{BenKnudsen:2018vu}.

The natural question of whether the ham sandwich theorem extends to convex partitions of $\rr^d$ leads to the following theorem.

\begin{theorem}\label{thm-convexequipartitionplane}
	Let $k, d$ be positive integers.  Given any $d$ absolutely continuous probability measures $\mu_1, \dots, \mu_d$ in $\rr^d$ there exists a convex partition of $\rr^d$ into $k$ parts $C_1, \dots, C_k$ such that
	\[
	\mu_i (C_j) = \frac{1}{k} \qquad \text{for } i=1, \dots, d,\quad  j=1,\dots, k.
	\]
\end{theorem}

The case $d=2$ was proved independently three times \cite{Ito:1998eb, Bespamyatnikh:2000tn, Sakai:2002vs}, and generalizes earlier results on ``perfect partitions of a cake'' \cite{cccg98-akiyama-perfect, Akiyama:2000uz}.  The general case also has three different proofs \cite{Soberon:2012kp, Karasev:2014gi, Blagojevic:2014ey}. The proof for the planar case by Bespamyatnikh, Kirkpatrick, and Snoeyink involves the discrete version of this result, and gives an algorithm to construct the partition in $O(N^{4/3}\log N \log k)$ time where $N$ is the total number of points to split.  If the point set is contained in a polygonal region (not necessarily convex), similar results can be obtained \cite{Bereg:2006jo}.

Theorem \ref{thm-convexequipartitionplane} can be applied to congressional district drawing \cite{Humphreys:2011iy, Soberon:2017kt} in the context of gerrymandering, which is related to the applications of the ham sandwich theorem to voting theory \cite{Cox:1984is}.  Another application is the following extension of Example \ref{example-russian}.

\begin{example}
We are given $n$ baskets, each containing a finite (possibly zero) amount of $q$ different types of fruits and a positive integer $k$.  Any basket may have a positive amount of more than one fruit.  It is possible to choose no more than $n/k + (k-1)^2q/k$ baskets and obtain at least a $(1/k)$-fraction of the total amount of each kind of fruit.
\end{example}

\subsection{The Nandakumar--Ramana-Rao problem}

Even though Theorem \ref{thm-convexequipartitionplane} deals with fair partitions of measures, the proofs of Karasev, Hubard, and Aronov, and of Blagojević and Ziegler yield much more \cite{Karasev:2014gi, Blagojevic:2014ey}.  They were motivated by a question of Nandakumar and Ramana Rao, which asked if every polygon in the plane could be split into $k$ convex parts of equal area and equal perimeter, for any positive integer $k$.  If the polygon has $n$ vertices, there are algorithms that find such a partition for $k=2^h$ in $O((2n)^h)$ time \cite{Armaselu:2015fz}.  

The perimeter is not a measure, but it is a continuous function on all compact convex sets under the Hausdorff metric.  In the plane, the first non-trivial case of the Nandakumar--Ramana-Rao problem to be solved was $k=3$ \cite{Barany:2010ke}.  The result stated below settled the problem for $k$ a prime power.  It is nicknamed ``the spicy chicken theorem''.

\begin{theorem}[Karasev, Hubard, Aronov 2014 and Blagojević, Ziegler 2014]\label{thm:spicy}
	Let $d$ be a positive integer and $k$ be a prime power.  Let $\mu$ be an absolutely continuous probability measure in $\rr^d$, let $f_1, \dots, f_{d-1}$ be $d-1$ continuous functions from the space of all closed convex sets in $\rr^d$ to $\rr$.  Then, there exists a convex partition $C_1, \dots, C_k$ of $\rr^d$ into $k$ sets such that
	\begin{alignat*}{2}
		\mu(C_j) &= \frac{1}{k} && \qquad \text{ for } j=1,\dots, k, \\
		f_i(C_j) &= f_i(C_{j'}) && \qquad \text{ for } i=1,\dots, d-1, \quad j,j' =1, \dots, k.
	\end{alignat*}
\end{theorem}

When the $d-1$ continuous functions $f_i$ are induced by measures (i.e., $f_i(K) = \mu_i (K)$), a simple subdivision argument yields Theorem \ref{thm-convexequipartitionplane}.  Recently, Akopyan, Avvakumov, and Karasev settled the Nandakumar--Ramana-Rao problem affirmatively for any $k$ \cite{Akopyan:2018tr}.  Their high-dimensional theorem, stated below, also implies Theorem \ref{thm-convexequipartitionplane}.

\begin{theorem}
	Let $d,k$ be positive integers.  Let $\mu_1, \dots, \mu_{d-1}$ be absolutely continuous probability measures in $\rr^d$, let $f$ be a continuous function from the space of all closed convex sets in $\rr^d$ to $\rr$, and let $k$ be a positive integer.  Then, there exists a convex partition $C_1, \dots, C_k$ of $\rr^d$ into $k$ sets such that
	\begin{alignat*}{2}
		\mu_i(C_j) &= \frac{1}{k} && \qquad \text{for } i=1,\dots, d-1, \quad j=1,\dots, k, \\
		f(C_j) &= f(C_{j'}) && \qquad \text{for } j, j' =1, \dots, k.
	\end{alignat*}
\end{theorem}

Avvakumov and Karasev extended closely related techniques to cover broader families of cake-cutting problems \cite{Avvakumov:2020oi}.  A result related to Theorem \ref{thm:spicy} is relevant in the proof of the symmetric case of Mahler’s conjecture in $\rr^3$. Iriyeh and Shibata proved implicitly the following theorem, which was stated precisely later by Fradelizi, Hubard, Meyer, and Roldán-Pensado \cite{Iriyeh:2020sy,Fradelizi:2019eq}.

\begin{theorem}
    Let $K$ be a convex body in $\rr^3$ symmetric respect to the origin. There are planes $H_1,H_2$, and $H_3$ through the origin that split $K$ into eight pieces of equal volume and such that each planar convex body $K\cap H_i$ is split into four parts of equal area by the other two planes.
\end{theorem}

As done previously, the planar sections' areas may be replaced by a continuous function on the planes through the origin. It seems difficult to extend this theorem to higher dimensions. This is because the natural generalizations fail to give an adequate setting for the test map/configuration space scheme.

\subsection{The Holmsen--Kynčl--Valculescu conjecture}

In the results above, the number of measures to be split among the parts is equal to the dimension.  This is optimal if we want to split each measure perfectly among each part of the partition.  If we have more measures, Holmsen, Kynčl, and Valculescu made the following conjecture for finite sets of points \cite{Holmsen:2017de}.

\begin{conjecture}[Holmsen, Kynčl, Valculescu 2017]\label{conjecture-holmsen}
	Let $d, \ell, m, k$ be integers such that $m \ge d \ge 2$ and $\ell \ge d$.  Suppose we are given a set of $\ell k$ points in $\rr^d$ in general position, each colored with one of $m$ colors.  If there exists a partition of them into $k$ sets of size $\ell$, each with points of at least $d$ different colors, then there also exists such a partition for which the convex hulls of the parts are pairwise disjoint.
\end{conjecture}

The case $d=2$ was proved in the same paper where the conjecture was stated.  The case $m=d$ was proved by Blagojević, Rote, Steinmeyer, and Ziegler \cite{Blagojevic:2019gb}.  They proved a stronger statement: a discrete version of Theorem \ref{thm-convexequipartitionplane} in high dimensions.  The discrete version in the case $m=d=2$ had been proved earlier \cite{Ito:1998eb, Bespamyatnikh:2000tn}.

The case $\ell = m = d$ of Conjecture \ref{conjecture-holmsen} is a known consequence of the ham sandwich theorem due to Akiyama and Alon \cite{Akiyama:1989fk}.  In particular, the planar case of this result implies that \textit{for any set of $k$ red points and $k$ blue points in the plane, there exists a perfect red-blue matching whose edges induce pairwise disjoint segments.}  For numerous extensions to non-crossing geometric graphs in the plane, we recommend the survey by Kano and Urrutia \cite{Kano:2020ws}.

Kano and Kyn\v{c}l's ``hamburger theorem'' implies the case $m=d+1, \ell=d$ of Conjecture \ref{conjecture-holmsen}.  The hamburger theorem describes mass partition results for $d+1$ measures using a single hyperplane \cite{Kano:2018dp}.

Several continuous analogues of Conjecture \ref{conjecture-holmsen} were proved by Blagojević, Palić, Soberón, and Ziegler \cite{Blagojevic:2019hh}, such as the following result. 

\begin{theorem}
	Let $d\geq 2$, $m\geq 2$, $k\geq 2$, and $c\geq d$ be integers.
If 
\[
m=n(c-d)+d,
\] 
then for every $m$ positive finite absolutely continuous measures on $\rr^d$, there exists a partition of $\rr^d$ into $k$ convex subsets $(C_1, \dots, C_k)$ such each of the subsets has positive measure with respect to at least $c$ of the measures.
\end{theorem}

The value $m=n(c-d)+d$ is optimal if additional constraints on the partition are imposed, such as splitting fairly $d-1$ of the measures.  However, it would be interesting to know if the  result is optimal in general.

\subsection{Partitions and their transversals}\label{subsec:transversals}
If we are interested in partitions of $\rr^d$ into convex pieces, we can impose additional geometric conditions.  For any partition of $\rr^d$ into $2^d$ pieces using $d$ hyperplanes, no hyperplane can intersect the interior of all parts.  Even though it may be impossible to find fair partitions of a single measure using $d$ hyperplanes, Yao and Yao showed that the lack of a hyperplane transversal may be preserved \cite{Yao:1985hf} (see \cite{Lehec:2009ya} for a constructive proof).

\begin{theorem}[Yao, Yao 1985]
	Let $\mu$ be an absolutely continuous measure in $\rr^d$.  There exists a partition of $\rr^d$ into $2^d$ convex parts of the same $\mu$-measure such that no hyperplane intersects the interior of all of them.
\end{theorem}

The theorem above is motivated by its applications in geometric range queries (notably, half-space queries) as a way to pre-process data \cite{Agarwal:1999vl}.  The geometric conditions of the Yao--Yao theorem lead to interesting questions.

\begin{problem}
	Let $k, d$ be positive integers.  Find the smallest number $n$ such that the following holds.  For any finite absolutely continuous measure $\mu$ in $\rr^d$ there exists a partition of $\rr^d$ into $n$ convex parts of the same $\mu$-measure such that every hyperplane misses the interior of at least $k$ parts.
\end{problem}

For $k=2$, it is known that $n \le 3 \cdot 2^{d-1}$ \cite{RoldanPensado:2014cc}.  In the plane, for $k=1$ four parts are enough.  For $k=2$, Buck and Buck's \cite{Buck:1949wa} equipartition by three concurrent lines shows that at $n \le 6$, although equipartitions like the one in Figure \ref{fig:cobweb}(a) also work.  In those, we use two lines to split the measure in parts of size $1/3, 1/6, 1/3, 1/6$, and then apply the ham sandwich theorem on the pieces of size $1/6$.
Another solution for $k=2$ can be obtained using a partition by three lines, two of which are parallel \cite{Karasev:2016cn}.
For $k=3$, we can use Schulman's equipartition result with \textit{cobwebs} to get $n \le 8$.  A cobweb consists of two intersecting lines $\ell_1, \ell_2$ and four points $x_1, x_2, x_3, x_4$ (two in $\ell_1$, two in $\ell_2$) in convex position.  The four sides of the quadrilateral $x_1x_2x_3x_4$ and $\ell_1 \cup \ell_2$ divide the plane into eight regions.  Figure \ref{fig:cobweb}(b) shows an example.   Schulman showed that, for any absolutely continuous finite measure $\mu$ in the plane, we can find a cobweb that splits $\mu$ into eight equal parts \cite{Schulman:1993}.  The reader can verify that every line misses the interior of at least three regions.

\begin{figure}
    \centerline{\includegraphics{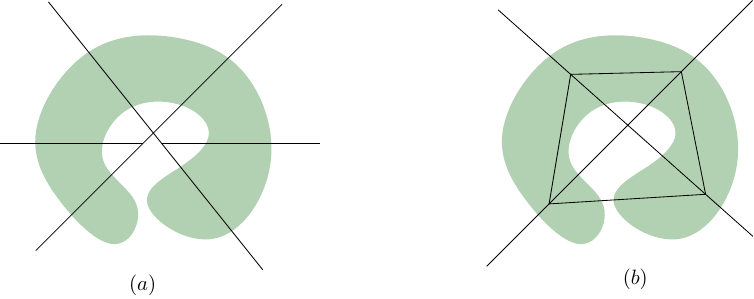}}
    \caption{(a) A partition of a measure induced by three lines into six equal parts. (b) A partition of a measure by a cobweb into eight equal parts.}
    \label{fig:cobweb}
\end{figure}

Another interesting question is to split more than one measure with similar geometric conditions.  We need to decrease the dimension of the transversal to have a meaningful question.

\begin{problem}
	Let $m \le d$ be positive integers.  Find the smallest $n$ such that the following holds.  For any $m$ absolutely continuous probability measure $\mu_1, \dots, \mu_{m}$ in $\rr^d$ there exists a convex partition $C_1, \dots C_n$ of $\rr^d$ such that
	\[
	\mu_i (C_j) = \frac{1}{n} \qquad \text{for } i=1,\dots, m \quad j= 1,\dots, n
	\]
	and every $(d-m)$-dimensional affine space of $\rr^d$ misses the interior of at least one $C_j$.
\end{problem}

The case $m=1$ is the Yao--Yao theorem, giving $n \le 2^d$.  The case $m=d$ is the ham sandwich theorem, giving $n = 2$.  It is tempting to conjecture that $n \le 2^{d+1-m}$.  However, this conjecture fails for $m = d-1$.  Any partition of $\rr^d$ into four convex sets such that each line avoids the interior of at least one part is made by the intersection of two hyperplanes.  Yet, the known bounds for Problem \ref{problem-gunbaum-hadwiger-ramos} show that at most $\frac{2}{3}d$ measures can be split by two hyperplanes into four equal parts, as opposed to the $d-1$ we would need for this problem.

If we do not require a perfect partition of our measures or point sets, then stronger partitioning results can be obtained.

\begin{theorem}[Matoušek 1992 \cite{Matousek:1992ed}]
	Let $n,k,d$ be positive integers, and let $X$ be a set of $n$ points in $\rr^d$.  For some $t=O(k)$, there exists a partition of $X$ into $t$ set $X_1, \dots, X_t$ such that
	\begin{itemize}
		\item For every $i = 1,\dots, t$ we have $\frac{n}{k} \le |X_i| \le \frac{2n}{k}$ and
		\item No hyperplane intersects the convex hull of more than $O(k^{1-1/k})$ parts.
	\end{itemize}
\end{theorem}

	Notice that we no longer ask that the convex hulls of the parts be pairwise disjoint.  This condition is replaced by the hyperplane avoiding condition.  The algorithms used to find such partitions are important for geometric range queries.  
	
	We may also be interested in avoiding a particular family of hyperplanes as transversals instead of avoiding any possible hyperplane.  Results such as the cutting lemma become important in this setting.  The cutting lemma was proved independently by Chazelle and Friedman \cite{Chazelle:1990fm} and by Matoušek \cite{Matousek:1990ke} (see also \cite{Chazelle:1993hv}).
	
\begin{theorem}[Cutting lemma in the plane]\label{theorem-cutting-lemma}
	Let $n,k$ be positive integers and $\mathcal{L}$ be a set of $n$ lines in the plane.  There exists a convex partition of the plane into $O(k^2)$ parts such that the interior of each part is intersected by at most $\frac{n}{k}$ lines of $\mathcal{L}$.  Moreover, each part of the partition is the intersection of three half-planes.
\end{theorem}

In the result above, the boundary of each part of the partition is contained in the union of at most three lines.  The cutting lemma is also prominent due to its application to incidence problems \cite{Clarkson:1990be}, such as the Szemerédi--Trotter theorem.

Another family of partitioning results which are useful for incidence problems are polynomial partitioning theorems.  One of the most notable examples is the following theorem that Guth and Katz used to find a near-optimal bound for the Erdős distinct distance problem \cite{Guth:2015tu}.

\begin{theorem}[Polynomial partitioning]\label{theorem-poly-partitioning-guthkatz}
	Let $k,d$ be positive integers and let $S$ be a finite set of points in $\rr^d$.  There exists a polynomial surface $Z$ of degree $O(2^{k/d})$ such that its complement $\rr^d\setminus Z$ is the union of $2^{k}$ open cells, each containing at most $|S|/2^k$ points of $S$.
\end{theorem}

There is now a wide variety of polynomial partitioning results.  A lucid introduction to the subject can be found in Guth's book \cite{Guth:2016wo}.  Such methods, and extensions of the cutting lemma, continue to be used successfully to prove results in incidence geometry \cite{Solymosi:2012kd, Zahl:2015cy} and harmonic analysis \cite{Guth:2016cy, Du:2017ow, Guth:2019uh}.

Another important theorem in discrete geometry is Rado's centerpoint theorem \cite{Rado:1946ud}.  For a finite set $X$ in $\rr^d$, we say that $p$ is a \emph{centerpoint} if every closed half-space that contains $P$ contains at least $|X|/(d+1)$ points of $X$.  The existence of centerpoints follows from Helly's theorem.  They can be used as a high-dimensional analogue of a median, so their computation is an interesting problem \cite{Chan:2004hn}.

A common generalization of the centerpoint theorem and the ham sandwich theorem was proved by Dolnikov \cite{Dolnikov:1992ut} and by Živaljević and Vrećica \cite{Zivaljevic:1990do}.

\begin{theorem}[Central transversal theorem]
Let $k, d$ be non-negative integers such that $k \le d-1$.  For any set of $k+1$ absolutely continuous probability measures $\mu_1, \dots, \mu_{k+1}$ there exists a $k$-dimensional affine space $V$ such that for any closed half-space $H$ with $V \subset H$ we have
\[
\mu_i(H) \ge \frac{1}{d-k+1}.
\]
\end{theorem}

For $k=0$ this is the centerpoint theorem and for $k=d-1$ this is the ham sandwich theorem.  A discrete variant of the theorem above, which simultaneously generalizes the ham sandwich theorem and Tverberg's theorem, was conjectured by Tverberg and Vrećica \cite{Tverberg:1993ia}.  Tverberg's theorem guarantees the existence of partitions of finite sets in $\rr^d$ into parts whose convex hulls intersect (see, e.g., \cite{Blagojevic:2017bl, Barany:2018fy, DeLoera:2019jb} and the references therein.)

Finally, we can impose conditions on the transversals to the support of our measures.  We say that a family of measures $\mu_1, \dots, \mu_n$ in $\rr^d$ is \textit{well-separated} if for any two non-empty disjoint subsets $I, J \subset[n]$, the support of the measure $\{\mu_i : i \in I\}$ can be separated by a hyperplane from the support of the measures $\{\mu_i : i \in J\}$.  B\'ar\'any, Hubard, and Jer\'onimo proved that the ham sandwich theorem can be significantly strengthened for well-separated families of measures \cite{Barany:2008vv}.

\begin{theorem}
Let $\mu_1, \dots, \mu_d$ be $d$ well-separated, absolutely continuous probability measures in $\rr^d$.  Let $\alpha_1, \dots, \alpha_d$ be real numbers in $[0,1]$.  There exists a hyperplane $H$ such that the two half-spaces $H^+, H^-$ bounded by $H$ satisfy
\[
\mu_i (H^+) = \alpha_i, \quad \mu_i(H^-) = 1-\alpha_i \qquad \text{for }i=1,\dots, d.
\]
\end{theorem}

The case $d=3$ was proved earlier by Steinhaus \cite{Steinhaus1945}.

\subsection{Partitions of families of lines}\label{subsection-lines}

There are different ways to extend the ham sandwich theorem to obtain a partition for families of lines.  Given two non-parallel lines $\ell_1, \ell_2$ in the plane, consider $y = \ell_1 \cap \ell_2$.  We can arbitrarily assign distinct signs ($+,-$) to the two rays from $y$ that form $\ell_1$, and the same for $\ell_2$. Then, any other line $\ell$ can be assigned one of four pairs $(++, +-, -+, --$) depending on which rays of $\ell_1$ and $\ell_2$ it intersects.  Langerman and Steiger proved that one can always use the assignment just described to obtain equipartitions of a single family of lines in the plane \cite{Langerman:2003ig}.

\begin{theorem}\label{theorem-steiger-langermann}
    For any family $A$ of $n$ lines in $\rr^2$, no two of them parallel, there exist two lines $\ell_1, \ell_2$ such that at least $\left\lfloor n/4 \right\rfloor$ lines of $A$ are assigned each of the pairs $++, +-, -+, --$ as described above.
\end{theorem}

%It is natural to ask what happens if there are two or more families of lines.

%\begin{problem}
%Let $n$ be a positive integer.  Find the smallest value $f(n)$ such that the following statement holds.  For any two families $A, B$ of lines in the plane, there exist two lines $\ell_1, \ell_2$ such that at least $f(n)$ lines of $A$ are assigned each of the pairs $++, +-, -+, --$ as described above, and likewise for $B$.
%\end{problem}

Other partitioning results have positive outcomes for several families of lines.  Given a finite family $A$ of lines in $\rr^2$, no two of them parallel, and a set $K \subset \rr^2$, we can consider the size of $K$ as
\[
\mu_A(K) = \max \{|X|: X \subset A, \ \ell_1 \cap \ell_2 \in K \text{ for all }\ell_1, \ell_2 \in X\}.
\]

Note that $\mu_A$ is not a measure.  Dujmović and Langerman proved a ham sandwich theorem for these functions \cite{Dujmovic:2013bi}.

\begin{theorem}\label{theorem-dujmovic-lines}
	Let $A, B$ be two finite families of lines in $\rr^2$, no two of them parallel.  There exists a convex partition of the plane into two halfplanes  $C_1, C_2$ such that 
	\begin{alignat*}{2}
		\mu_A (C_i) &\ge \sqrt{|A|} && \qquad\text{for } i=1,2, \\
		\mu_B (C_i) &\ge \sqrt{|B|} && \qquad\text{for } i=1,2.
	\end{alignat*}
\end{theorem}

The lower bounds on $\mu_A (C_i)$ and $\mu_B(C_i)$ are optimal in the result above.

An extension of the theorem above similar to Theorem \ref{thm-convexequipartitionplane} was proved by Xue and Soberón \cite{SoberonXue:2019}.

\begin{theorem}
	Let $A, B$ be two finite families of lines in $\rr^2$, no two of them parallel, and $k$ be a positive integer.  There exists a convex partition of the plane into $k$ parts  $C_1, C_2, \dots, C_k$ such that 
	\begin{alignat*}{2}
		\mu_A (C_i) &\ge k^{\ln(2/3)}{|A|}^{1/k}-2k && \qquad\text{for } i=1,\dots, k, \\
		\mu_B (C_i) &\ge k^{\ln(2/3)}{|B|}^{1/k}-2k && \qquad\text{for } i=1,\dots, k.
	\end{alignat*}
\end{theorem}

A key step in the proof by Langerman and Dujmović is the following, which is a consequence of the Erdős--Szekeres theorem on monotone sequences.

\begin{lemma}\label{lemma-erdos-szekeres}
	Let $A$ be a finite family of lines in $\rr^2$, no two of them parallel.  Let $(C_1, C_2)$ be a convex partition of the plane into two parts.  Then,
	\[
	\mu_A(C_1)\mu_A(C_2) \ge |A|.
	\]
\end{lemma}

Theorem \ref{theorem-dujmovic-lines} extends to higher dimensions.  The guarantee on the size of each half-space relies on geometric Ramsey-type results that have a much smaller rate of growth \cite{Conlon:2014fx}. An interesting question is whether Lemma \ref{lemma-erdos-szekeres} generalizes to a larger number of parts.

\begin{problem}
	Let $A$ be finite family of lines in $\rr^2$, no two of them parallel.  Let $(C_1, C_2, C_3)$ be a convex partition of the plane into three parts.  Determine if the following inequality must hold:
	\[
	\mu_A(C_1)\mu_A(C_2) \mu_A(C_3) \ge |A|.
	\]
\end{problem}

The following result follows from yet another interpretation for splitting lines in the plane \cite{Bereg:2015cz}.

\begin{theorem}
	Let $A, B, C$ be three families of lines in the plane, each with $2n$ elements.  There exists a segment that intersects exactly $n$ lines of each set.
\end{theorem}

The following new theorem improves the result above. Recall that there is a natural correspondence between the set of hyperplanes in $\rr^d$ and the punctured projective space $\rr \mathrm P^d\setminus\{x_0\}$. Therefore we may talk about absolutely continuous measures in the space of hyperplanes by using this bijection.

%\textcolor{red}{Sigo sin estar 100\% contento con esto.}

\begin{theorem}
	Let $d$ be a positive integer.  There exists an integer $m = 2d-O(\log d)$ such that the following holds.  For any $m$ absolutely continuous probability measures $\mu_1,\dots,\mu_{m}$ in the space of hyperplanes of $\rr^d$ there exists a segment or ray $s$ in $\rr^d$ such that for all $i=1,\dots, m$ we have
	\[
	\mu_i \left( \{H \text{ hyperplane}: H \cap s \neq \emptyset \} \right) =1/2.
	\]
\end{theorem}

\begin{proof}
	We apply point duality to in $\rr^d$, so each measure is now a measure of $\rr^d$.  By the stronger version of Theorem \ref{theorem-weak} (see \cite{Blagojevic:2018ve}), we can find two hyperplanes whose chessboard coloring halves each measure.  The component of this coloring that does not contain the origin corresponds to the set $s$ we were looking for.
\end{proof}

%Even though Theorem \ref{theorem-weak} gives a weaker bound, it shows that $m \ge d+1$ for all values of $d$ in the theorem above.

Another way to split families of lines appears if we go one dimension higher.  Given two non-vertical lines $\ell_1, \ell_2$ in $\rr^3$ that don't intersect but are not parallel, there is a unique vertical line $h$ that intersects both of them.  We say that $\ell_1$ is above $\ell_2$ if $\ell_1 \cap h$ is above $\ell_2 \cap h$.  Schnider showed that we can find a ham sandwich theorems for splitting several families of lines with a single additional line \cite{Schnider:2020kk}.

\begin{theorem}
	Let $A, B, C$ be three finite families of lines in $\rr^3$ so that no two of them intersect, no two of them are parallel, and no line is vertical.  Moreover, assume $|A|, |B|, |C|$ are even.  Then, there exists a line $\ell$ that is above exactly $|A|/2$ lines of $A$, $|B|/2$ lines of $B$ and $|C|/2$ lines of $C$.
\end{theorem}

The partitioning line can be found in $O(n^2\log^2 n)$ time for $n = |A|+|B|+|C|$ \cite{Pilz:2019tg}.  If we let $H$ be the vertical plane that contains $h$, almost every line in $A$, $B$, $C$ intersects $H$ at a single point, and $h$ is now a halving line for each of those sets.  Therefore, the theorem above can be interpreted as an improvement of the ham sandwich theorem: if we are given the freedom to choose $H$, we can halve more colors than the usual ham sandwich theorem usually allows.  Schnider's result extends to halving families of $k$-dimensional affine planes in $\rr^d$ using a $(d-k-1)$-dimensional affine plane to split them.

\subsection{Partitions with restrictions on the pieces}

Convex partitions of $\rr^d$, even restricted to generalized Voronoi diagrams, are very general.  We can impose additional constraints on the possible shapes of the pieces.  In contrast to the rest of the survey, this subsection's main results are about types of partitions that can \emph{never} split certain measures evenly.

Buck and Buck proved one of the first theorems of this type. They showed that it is impossible to split a convex body in $\rr^2$ into seven regions of equal area by three lines \cite{Buck:1949wa}. Scott later generalized this to higher dimensions \cite{Scott:1990}.

\begin{theorem}
For $d\ge 2$, no compact convex body in $\rr^d$ can be split by $d+1$ hyperplanes into $2^{d+1}-1$ parts of equal volume.
\end{theorem}

The following problem was solved by Monsky \cite{Monsky:1970fi}, answering a question of Fridman \cite{Thomas:1968is}.

\begin{theorem}
Let $k$ be an odd integer.  There exists no partition of a square into $k$ triangles of the same area.
\end{theorem}

The techniques used to prove this are surprising, as they seem at first glance far detached from this problem.  The first tool is Sperner's lemma \cite{Sperner:1928ke}.  Sperner's lemma guarantees the existence of colorful simplices on certain colorings of triangulations of polytopes.  It is a discrete version of the Knaster--Kuratowski--Mazuriewicz theorem \cite{Knaster:1929vi}.  The second is $p$-adic valuations of real numbers.  For a prime number $p$, the $p$-adic valuation of a rational number $x=p^\alpha \left( \frac{a}{b}\right)$ with where $a,b$ are not divisible by $p$ is $|x|_p = p^{-\alpha}$.  Such valuation can be extended to the real numbers.  Monsky's proof uses $p=2$.  Then, the points in the plane are divided into three (not convex!) sets
\begin{align*}
    A_1 & = \{(x,y) : |x|_2 < 1,\ |y|_2<1\}, \\
    A_2 & = \{(x,y) : 1 \le |x|_2,\ |y|_2<|x|_2\}, \\
    A_3 &= \{(x,y) : |x|_2 \le  |y|_2,\ 1\le |y|_2\}.
\end{align*}
The rest of the proof consists of showing that $A_1, A_2, A_3$ induce a KKM coloring on any triangulation of the square $[1,2]\times[1,2]$, and that any triangle with one vertex in each of $A_1, A_2, A_3$ cannot have area of the form $1/k$ for any odd integer $k$.

We can obtain a different perspective on this solution by considering the function
\begin{align*}
    f:\rr^2 & \to \rr^2 \\
    (x,y) & \mapsto (|x|_2, |y|_2).
\end{align*}
A partition of the domain into three particular convex cones induces the partition of $\rr^2$ used in Monsky's solution.

The theorem above can be improved upon.  Kasimatis showed that \textit{if a regular $n$-gon is divided into triangles of equal area and $n\ge 5$, their number must be a multiple of $n$} \cite{Kasimatis:1989gn}.  Similar results are known for broader families of polygons \cite{Rudenko:2013cp}.

A high-dimensional version says that \textit{a hypercube in $\rr^d$ cannot be partitioned into simplices of equal volume unless their number is a multiple of $d!$} \cite{Mead:1979bq, Bekker:1998fz}.  The proof of these results uses $p$-adic valuations for every prime $p$ that divides $d!$.

If instead of triangles we use congruent convex pieces, the following problem remains open.

\begin{problem}
    Let $p$ be a prime number.  Determine if the only way to partition a square into $p$ congruent convex sets is by using $p-1$ cuts parallel to one of the rectangle's sides.
\end{problem}

This has been answered affirmatively for $p=3$ \cite{Maltby:1994ww} and for $p=5$ \cite{Yuan:2016ic}.

\section{More classes of partitions}

%Mass partitioning results described in the previous sections use partitions with natural parametrization spaces.  In this section we describe partition mass partitioning results which use very different partitioning sets.

\subsection{Sets of fixed size}
Most mass partition problems deal with ways to split measures into pieces of equal size.  Often, this is a strict requirement on the problem.  For example, it's easy to find pairs of probability measures in $\rr^d$ such that for any $\alpha \neq 1/2$ there exist no hyperplane that cuts simultaneously $\alpha$ from both measures on one side and $1-\alpha$ on the other.

For $\alpha \le 1/2$ it may still be possible to find a single convex set $K$ of the same size $\alpha$ under many measures.

\begin{problem}\label{prob:one-set-fixed-size}
	Let $\alpha \in (0,1/2]$ and let $d$ be a positive integer.  Determine if, for any $d$ absolutely continuous probability measures $\mu_1, \dots, \mu_d$ in $\rr^d$, it is possible to find a convex set $K$ such that
	\[
	\mu_i (K) = \alpha \qquad \text{for } i=1,\dots, d.
	\]
\end{problem}

The problem above is trivial for $d=1$.  It has been solved positively for $d=2$ with any $\alpha$ by Blagojević and Dimitrijević Blagojević \cite{Blagojevic:2007ij} (see also \cite{Bespamyatnikh:2003hp} for a proof for $\alpha < 1/3$ using fewer topological tools).  For the discrete version of this problem, if we are given $n$ points of two different colors in the plane, there are algorithms that find a convex set $K$ with an $\alpha$-fraction of each color in $O(n^4)$ time in general and in $O(n\log n)$ time when $\alpha < 1/3$ \cite{Aichholzer:2018gu}.

If $1/\alpha$ is an integer, Akopyan and Karasev proved a stronger statement: there is a convex set $K$ of size $\alpha$ for $d+1$ measures given in advance \cite{Akopyan:2013jt}.  For $d+1$ measures the condition on $\alpha$ is necessary.

In dimension one, a classic result of Stromquist and Woodall solves the problem of finding a simple set of size $\alpha$ in many measures simultaneously \cite{Stromquist:1985tg}.

\begin{theorem}
Let $\alpha \in [0,1]$.  For any family of $m$ absolutely continuous probability measures $\mu_1, \dots, \mu_m$ in $\rr$, there exists a set $K$ that is the union of at most $m$ intervals such that
\[
\mu_i (K) = \alpha \qquad \text{for } i = 1, \dots, m.
\]
\end{theorem}

Another variant of Problem \ref{problem-gunbaum-hadwiger-ramos}, posed by Grünbaum and made public by Bárány, concerns uneven partitions \cite{Katona:2010op}.

\begin{problem}
	Let $K$ be a compact convex set in the plane of area one, and $t \in (0,1/4]$.  Determine if it is always possible to find two orthogonal lines in the plane that split $K$ into four regions of areas $t,t,1/2-t,1/2-t$ in clockwise order.
\end{problem}

The answer to this problem is conjectured to be positive. This is true when the diameter of $K$ is at least $\sqrt{37}$ times larger than the minimum width of $K$ \cite{Arocha:2010ug}. The question can also be asked for an absolutely continuous measure instead of a convex body.  In this case, Bárány conjectures that the answer is negative.

Blagojević and Dimitrijević Blagojević proved the following version for the sphere \cite{Blagojevic:2013ja}.
\begin{theorem}
    For any absolutely continuous probability measure $\mu$ on $S^2$ and any $t\in[0, 1]$ there are four great semi-circles emanating from a point $x\in S^2$ that split $S^2$ into angular sectors $\sigma_1,\dots,\sigma_4$, in clockwise order, such that $\mu(\sigma_1)=\mu(\sigma_4)=t$, $\mu(\sigma_2)=\mu(\sigma_3)=1/2-t$ and the angles formed by each sector satisfy $\angle(\sigma_1)=\angle(\sigma_4)$, $\angle(\sigma_2)=\angle(\sigma_3)$.
\end{theorem}

\subsection{Partitions by fans and cones}\label{subsec:fans-cones}

In the plane, a ham sandwich cut of two measures is given by a line.  Another simple shape we can use to make a partition is a $k$-fan, consisting of $k$ rays emanating from a point.  We call the parts \textit{wedges} in such a partition.  We distinguish convex $k$-fans, where each of the $k$ resulting parts must be convex, from general $k$-fans, that admit up to one non-convex part.  We admit degenerate cases, formed by a partition into $k$ sets by $k-1$ parallel line in the case of convex $k$-fans, and a partition into $k$ sets by $k$ parallel lines (where the two extreme sections are part of the same set) for the non-convex fans.  Just as the space of half-planes can be parametrized with $S^2$, certain spaces of $k$-fans lead us to interesting topological parametrizations.  Given an absolutely continuous probability measure $\mu$ in the plane, and positive real numbers $\alpha_1, \dots, \alpha_k$ that sum to one, the space of  $k$-fans such that the wedges have measures $\alpha_1, \dots, \alpha_k$ in clockwise order can be parametrized with $SO(3)$.  If $\alpha_1 = \dots = \alpha_k = 1/k$, a shift over the wedges induces a free action of $\zz_k$ on this space.

We say that a family of measures in the plane can be $(\alpha_1, \dots, \alpha_k)$-partitioned by a $k$-fan if there exists a $k$-fan such that the $i$-th wedge has an $\alpha_i$-fraction of each of the measures.

\begin{problem}\label{prob:fan-partitions}
    Let $(\alpha_1, \dots, \alpha_k)$ be a $k$-tuple of positive real numbers whose sum is one.  Determine the largest number $m$ such that any $m$ absolutely continuous measures in the plane can be simultaneously $(\alpha_1,\dots, \alpha_k)$-partitioned by a $k$-fan.
\end{problem}

\begin{table}
    \centering
    \caption{Known results for simultaneous fan partitions in the plane.}
    \begin{tabular}{|c|c|c|c|}
        \hline
        & Two measures & Three measures & Reference\\
        \hline
        Two-fans & $(\alpha, 1-\alpha)$ for $0 < \alpha \le 1/2$ & $(1/2,1/2)$ & \cite{Barany:2001fs} \\
        Three-fans & $(\alpha, \alpha, 1-2\alpha)$ for $0< \alpha < 1/2$ & & \cite{Blagojevic:2007ij}\\
        %& $(1/2, 1/4, 1/4)$ & \\
        %& $(1/3,1/3,1/3)$ & \\
        Four-fans & $(1/4,1/4,1/4,1/4)$ & & \cite{Barany:2002tk} \\
        & $(1/5,1/5,1/5,2/5)$ & & \cite{Barany:2002tk} \\
        \hline
    \end{tabular}
    \label{tab:fan-paritions}
\end{table}

The known results for Problem \ref{prob:fan-partitions} are summarized in Table \ref{tab:fan-paritions}.   Algorithms to find $(1/2,1/2)$-partitioning of three point sets in $O(n^2\log n)$ time were found by Bereg, where $n$ is the total number of points \cite{Bereg:2005vo}.  The existence of $(\alpha, \alpha, 1-2\alpha)$-partitions for two measures by Blagojević and Dimitrijević Blagojević \cite{Blagojevic:2007ij} implies a positive answer for Problem \ref{prob:one-set-fixed-size} in the plane.  This is because at least one of the sections with size $\alpha$ must be convex.  Theorem \ref{thm-convexequipartitionplane} shows that, for $(1/3,1/3,1/3)$ fan partitions of two measures, we can also assume the fan is convex.  Further results can be found in \cite{Barany:2013fd, Vrecica:2003gt, Barany:2010ke}.

Another variation concerning partitions by translates of fans in the plane was proved by Balitskiy, Garber, and Karasev \cite{Balitskiy:2015fy}.
%This extends Theorem \ref{theorem-kuratowski-steinhaus} to more measures.

\begin{theorem}\label{thm:BGK15}
Let $k\ge 3$ be an odd integer and $F_1,\dots, F_k$ be a $k$-fan in the plane.  For any two absolutely continuous probability measures $\mu_1, \mu_2$, there exists an index $i$ and a vector $x$ in the plane such that
\[
\mu_j (F_i) = \frac{1}{2} \qquad \text{for } i=1,2.
\]
If $k$ is even, the statement above still holds if the fan is made by $k/2$ concurrent lines.
\end{theorem}

%Some cases hold in higher dimensions.  For example, in $\rr^d$ it is possible to find a simultaneous $(1/4,1/4,1/2)$ fan partition for any $d$ measures.  It suffices to apply the ham sandwich theorem, and then apply again the ham sandwich theorem on one of the sides.  The intersection of 

%Bárány and Matoušek proved that any two measures could be simultaneoulsy $(\alpha, 1-\alpha)$-partitioned by a $2$-fan, $(1/2,1/4,1/4)$-partitioned by a $3$-fan, $(1/4,1/4,1/4,1/4)$-partitioned by a $4$-fan and $(1/5,1/5,1/5,2/5)$-partitioned by a $4$-fan \cite{Barany:2001fs, Barany:2002tk}.  The existence of $(1/3,1/3,1/3)$-equipartitions with convex $3$-fans follows from Theorem \ref{thm-convexequipartitionplane} in the plane.  

In high dimensions, there are two common ways to extend the notion of a fan.  One is to look for partitions of $\rr^d$ formed by projecting to a $2$-dimensional subspace, finding a $k$-fan partition, and then taking the inverse image of each section under the projection.  For these partitions, Makeev showed in 1994 that one can split $\lfloor (2d-1)/(p-1)\rfloor +1$ measures with a $p$-fan, where $p$ is prime (see \cite[Thm 57]{karasev:2008fj}).  A full proof was recently presented by Schnider \cite{Schnider:2019ua}.  

Bukh, Matou\v{s}ek and Nivasch proved that any finite absolutely continuous measure $\mu$ in $\rr^d$ can be split into $4d-2$ equal parts using $2d-1$ hyperplanes with a common $(d-2)$-dimensional affine plane \cite{Bukh:2010hz}.  As an application of this result, they show that for any set $S$ of $n$ points in $\rr^d$ there exists a $(d-2)$-dimensional plane that intersects the convex hull of at least $(1/24)(1-1/(2d-1)^2)n^3 - O(n^2)$ triangles with vertices in $S$.

Several results listed in Table \ref{tab:fan-paritions} extend to $\rr^d$.  Consider the case of $(1/4,1/4,1/2)$ partitions by a $3$-fan.  For any $d$ measures in $\rr^d$, by the ham sandwich theorem we can find a hyperplane $H_1$ halving every measure.  If we apply the ham sandwich theorem again on one side $H_1^+$ of the partition, we find a hyperplane $H_2$ halving each measure in $H_1^+$.  The intersection of these two hyperplanes is a $(d-2)$-dimensional affine space.  The set $H_1 \cup (H^+_1 \cap H_2)$ forms a $3$-fan that shows a simultaneous $(1/4,1/4,1/2)$ partition.

For the case of $(1/3,1/3,1/3)$-partitions using $3$-fans, we can also simultaneously split any $d$ measures in $\rr^d$.  It suffices to apply Theorem \ref{thm-convexequipartitionplane} for $k=3$.  Any partition of $\rr^d$ into three convex sets is made by a $3$-fan.

The second approach is to take a convex polytope $P$ surrounding the origin and consider the cones induced by the facets of the $P$.  Formally, the \textit{cone partition} induce by $P$ is the partition formed by the family of sets
\[
\{\operatorname{cone}(K): K \text{ is a facet of } P\}.
\]
We also consider conical partitions that are not centered at the origin.  As an example, consider an absolutely continuous probability measure $\mu$ in $\rr^d$.  We want to know if there exists a regular hypercube $P$ such that each of the $2d$ cones it induces have the same $\mu$ measure.  The case $d=3$ was first solved by Makeev \cite{makeev1988partitioning}.  If the measure is centrally symmetric around the origin and the hypercube we seek must also be centered at the origin, Karasev showed that such a cone partition exists for $d$ a power of a prime \cite{Karasev:2010kh, Karasev:2010cy}.

If we fix a cone partition induced by a simplex whose interior contains the origin, any absolutely continuous probability measure in $\rr^d$ can be split into $d+1$ sets of prescribed size by a translate of this cone partition, as shown by Kuratowski and Steinhaus \cite{kuratowski1953application}, extending earlier results by Levi \cite{Levi:1930ea}.

\begin{theorem}\label{theorem-kuratowski-steinhaus}
Let $\Delta$ be a simplex in $\rr^d$ whose interior contains $0$ and let $K_1, \dots, K_{d+1}$ be the facets of $\Delta$.  Let $\mu$ be an absolutely continuous probability measure in $\rr^d$ and $\alpha_1, \dots, \alpha_{d+1}$ be positive real numbers such that $\alpha_1 + \dots + \alpha_{d+1} = 1$.  Then, there exists a vector $x \in \rr^d$ such that
\[
\mu_i \Big( x + \operatorname{cone}(K_i)\Big) = \alpha_i \qquad \text{for }i=1,\dots,d+1.
\]
\end{theorem}

\begin{figure}
\centerline{\includegraphics[width=\textwidth]{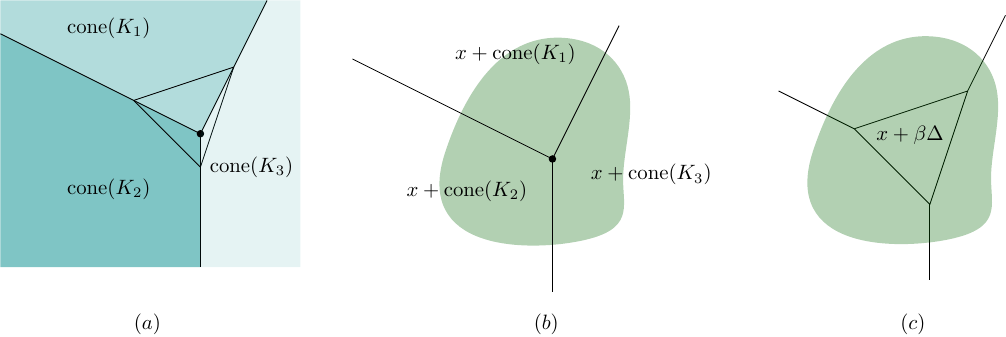}}
    \caption{(a) Three cones defined by a triangle $\Delta$ with the origin in its interior in $\rr^2$. (b) Any probability measure $\mu$ can be split into pieces of size $\alpha_1, \alpha_2, \alpha_3$ by a translate of this partition.  (c) Vrećica and Živaljević extended this result to four parts, where an additional section similar to $\Delta$ is included.}
    \label{fig:cones}
\end{figure}

This theorem was also proved by Borsuk \cite{borsuk1953application}, and later rediscovered by Vrećica and Živaljević \cite{Vrecica:2001dh, Vrecica:1992cx}.  Vrećica and Živaljević's approach extends to partitions into $d+2$ parts of prescribed sizes, with parts of the form
\[
\begin{cases}
x + \Big(\operatorname{cone}(K_i) \setminus (\beta \Delta ) \Big) & \text{for }i=1,\dots, d+1, \\
x+\beta \Delta & \text{as the last part.}
\end{cases}
\]
for some $\beta >0$ and $x \in \rr^d$.
A discrete version of Theorem \ref{theorem-kuratowski-steinhaus} in the plane was proved by Bose et al. \cite{bose1997floodlight} in order to solve an illumination problem. The translated fan they use to split a set of $n$ points does not have to be a convex fan, as in Kuratowski and Steinhaus' theorem.  Moreover, their method yields an algorithm that splits a set of points into three families of prescribed size in $O(n \log n)$ time.

%These algorithms improve upon earlier work by Numata and Tokuyama, and Tokuyama and Nakano, who developed randomized algorithms to compute such partitions when each of the three angles of the fan were equal to $2\pi/3$.  Tokuyama and Nakano also prove Theorem \ref{theorem-kuratowski-steinhaus} when $\Delta$ is a regular simplex in $\rr^d$ centered at the origin, and their algorithms cover the case of measures given by weighted points \cite{Tokuyama:1991ge, Tokuyama:1992ci}.

A simple proof of Theorem \ref{theorem-kuratowski-steinhaus} with modern techniques follows from the results of Aurenhammer, Hoffmann, and Aronov \cite{Aurenhammer:1998tj} mentioned in Section \ref{section-convexpartitions}.  It suffices to apply the main theorem of \cite{Aurenhammer:1998tj} to a power diagram with sites $K_i^*$ (the polar of the hyperplane containing $K_i$) for $i=1,\dots, d+1$.  To obtain the extension of Vrećica and Živaljević, just include an additional site at $0$.  It is easy to show that the partitions induced by these power diagrams are among those considered in Theorem \ref{theorem-kuratowski-steinhaus}.

A result very similar to Kuratowski and Steinhaus' theorem was proved by Numata and Tokuyama \cite{Numata:1993sp}.
\begin{theorem}\label{theorem-numata-tokuyama}
Let $\Delta$ be a simplex in $\rr^d$ and let $K_1, \dots, K_{d+1}$ be the facets of $\Delta$. Let $P$ be a set of $n$ points contained in the interior of $\Delta$ and $n_1, \dots, n_{d+1}$ be non-negative integers such that $n_1 + \dots + n_{d+1} = n$. Then, there exists a point $p\in\Delta$ such that
\[
|\operatorname{conv}(\{p\}\cup K_i)| = n_i \qquad \text{for }i=1,\dots,d+1.
\]
\end{theorem}
We can recover a discrete version of Theorem \ref{theorem-kuratowski-steinhaus} from this theorem by using limits involving increasingly larger homothetic simplices.
Numata and Tokuyama also give algorithms to find $p$ in $O(d^2n \log(n) + d^3n)$ time or in $O(n)$ time if $d$ is considered as a constant.

Cone partitioning problems allow for an approach with Fourier analysis if the cones are made with the facets of a polytope that is invariant under a representation of a finite group \cite{Simon:2015fg}.  Another extension of cone partitions, called polyhedral curtains, is presented by Živaljević to prove yet another generalization of the ham sandwich theorem \cite{Zivaljevic:2015im}.

\subsection{Partitions with curves of bounded complexity}

In addition to hyperplanes, other algebraic surfaces can be used to make a partition.  Perhaps the best-known result of this kind is the Stone--Tukey theorem \cite{Stone:1942hu}

\begin{theorem}
Let $d, r$ be positive integers.  Let $m = \binom{d+r}{r}-1$ and $\mu_1, \dots, \mu_m$ be $m$ finite absolutely continuous measures in $\rr^d$.  There exists a multinomial $f:\rr^d \to \rr$ of degree at most $r$ such that the two sets
\begin{align*}
    C_1 & = \{x \in \rr^d: f(x) \ge 0\} \\
    C_2 & = \{x \in \rr^d: f(x) \le 0\} 
\end{align*}
have the same size in each of the measures.
\end{theorem}

The proof involves the use of a standard Veronese map $f:\rr^d \to \rr^{\binom{d+r}{r}-1}$, where the $i$-coordinate of the image corresponds to the evaluation of the $i$-th non-trivial monomial of degree at most $r$ on $d$ variables.  The inverse image of a half-space in the higher-dimensional space corresponds to a set such as $C_1$ or $C_2$ in $\rr^d$.  Therefore, as in the proof we present of the ham sandwich theorem, the space of sets of the form $C_1$ can be parametrized by a high-dimensional sphere, and an application of the Borsuk--Ulam theorem finishes the proof.  Stone and Tukey observed that much more could be obtained using that same idea.  As an example, consider the following corollary.

\begin{corollary}
Let $d$ be a positive integer.  For any $d+1$ finite absolutely continuous measures in $\rr^d$, there exists a sphere containing exactly half of each measure.
\end{corollary}

The case $d=3$ was proved by Steinhaus independently \cite{Steinhaus1945}.  The polynomial partitioning theorem (Theorem \ref{theorem-poly-partitioning-guthkatz}) stands out as an important application of the Stone--Tukey theorem.

Another consequence of the proof method described above is that any $m$ measures in the plane can be simultaneously halved by the graph of a polynomial of degree $m-1$.  Instead of polynomials, we may want to use paths that can only use vertical or horizontal segments.  The following conjecture by Mikio Kano remains open.

\begin{conjecture}\label{conj:kano-paths}
Let $m$ be a positive integer.  For any $m$ absolutely continuous measures, there exists a curve in $\rr^2$ formed by horizontal and vertical segments, that takes at most $m-1$ turns, and splits each measure in half.
\end{conjecture}

%\textcolor{red}{Según yo hay un ejemplo sencillo de que es falsa la conjetura, era como con 7 puntos.}

The conjecture above has been solved for $m=2$ \cite{Uno:2009wk}.  If the horizontal portions of the paths are allowed to go through infinity, then the conjecture is known to hold \cite{Karasev:2016cn}. Theorem \ref{thm:BGK15} implies the case $m=2$ of Conjecture \ref{conj:kano-paths} by taking a fan made by the union of a vertical line and a horizontal line.  Figure \ref{fig:paths-kano} shows that we must consider self-intersecting paths for Conjecture \ref{conj:kano-paths} to be true.

\begin{figure}
    \centering
    \includegraphics{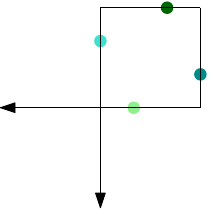}
    \caption{We can consider four measures in $\rr^2$, each concentrated near one of the points in the figure.  Any path made of vertical and horizontal segments, making at most three turns, that splits each measure must self-intersect.}
    \label{fig:paths-kano}
\end{figure}

\section{Extremal variants}

In the ham sandwich theorem, the halving hyperplane is sometimes unique.  For example, if each of the $d$ measures is uniformly distributed in a certain sphere, a halving hyperplane must contain all the spheres' centers.

If we reduce the number of measures, we expect to have significantly more halving hyperplanes.  This leads to extremal versions of ham sandwich problems.  We could be interested in counting halving hyperplanes, or finding halving hyperplanes with additional conditions.  Even the case of a single measure is interesting.

Consider the case when we are given a measure $\mu$ whose support is contained in a polygonal region $P$, which is not necessarily convex.  We can look for a segment between points in the boundary of $P$ that divides $\mu$ as evenly as possible.  It is possible that no chord of this type halves $\mu$ exactly.  We can also look for a broken line, formed by segments contained in $P$, that halves $\mu$. There are algorithms that find the shortest halving broken line \cite{Bose:2007eu}.

If we aim to split a fraction greater than some constant $\alpha$ on each side with a single segment in $P$, there are algorithms that compute the shortest chord possible \cite{Bose:1998fw}.  The running time depends on the number of segments in the boundary of $P$ and the complexity of $\mu$.  The maximum value $\alpha$ we can obtain depends on whether we want one of the chord's endpoints to be a vertex of $P$.

If the support of $\mu$ is a polygon $P$, we can also obtain algorithms for some continuous ham sandwich theorems.  These include halving lines of the area of two convex polygons or halving planes of the volume of three convex polyhedra \cite{Shermer:1992kh, Stojmenovic:1991id}.  The algorithms depend on the number of vertices of the supports.  These problems are closely related to a broader class of problems known as polygonal cutting \cite{Keil:1985bc}.

\subsection{Halving hyperplanes}

Given a finite set $S$ of points in general position in $\rr^d$, we say that a hyperplane is a \textit{halving hyperplane} if it cuts the set exactly in half.  We allow the hyperplane to contain a point of $S$ if $|S|$ is odd. The ham sandwich theorem says that any $d$ finite sets in $\rr^d$ share a halving hyperplane, which intuitively tells us that any finite set must have a large family of halving hyperplanes.  The question of counting the number of halving hyperplanes becomes relevant.  We say that two halving hyperplanes $H_1$ and $H_2$ of a set $S$ are equivalent if they split $S$ into the same pair of subsets.

In $\rr^3$, the problem of counting halving planes was the original motivation for studying the colorful variations of Tverberg's theorem \cite{Barany:1990wa}.  This gave an $O(n^{3-\varepsilon})$ bound on the number of halving planes for a set of $n$ points.  The same technique gave an $O(n^{d-\varepsilon})$ bound in $\rr^d$ \cite{Zivaljevic:1992vo, Alon:1992ek}.  The best lower bound is due to Tóth, who constructed sets with $n^{d-1}e^{\Omega(\sqrt{\log n})}$ halving hyperplanes \cite{Toth:2001ie}.  Nivash proved similar bounds for the plane but with a larger constant in the exponent \cite{nivasch2008improved}.  Several improvements in dimension three have pushed the upper bound down to $O(n^{5/2})$ \cite{Sharir:2001hf} and to $O(n^{4-1/18})$ in dimension four \cite{Sharir:2011tw}.  In dimension two, the current best bound of $O(n^{4/3})$ halving lines is due to Dey \cite{Dey:1998ib}.  A long-standing conjecture on the number of halving lines in the plane as made by Erdős, Lovász, Simmons, and Strausz \cite{Erdos:1973dy}.

\begin{conjecture}
	For every $\varepsilon > 0$, the number of halving lines for a set of $n$ points on the plane is $O(n^{1+\varepsilon})$.
\end{conjecture}

Additional structure on the set of points can help us improve the bounds on halving hyperplanes.  One such condition is \emph{$\delta$-density}.  A set of points in $\rr^d$ is $\delta$-dense if the ratio between the largest distance and the shortest distance does not exceed $\delta n^{1/2}$.  For $\delta$-dense sets of $n$ points in the plane, the number of halving lines is bounded by $O(n^{5/4})$.  In dimension $d \ge 3$, the number of halving hyperplanes of a dense set at most $O(n^{d-(2/d)})$ \cite{Edelsbrunner:1997hx}.  The current lower bound for the number of halving hyperplanes for a dense family of points match those for points in general \cite{Kovacs:2019kz}.  Improved bounds on the number of halving planes also exist for random sets of points \cite{Barany:1994da}.

Many of the results above extend to counting $k$-sets, which are relevant in many problems in computational geometry \cite{Cole:1984tu}.  Given a finite set $X \subset \rr^d$, a $k$-set $A \subset X$ is a set such that $|A|=k$ and $\conv (A) \cap \conv (X \setminus A) = \emptyset$ (alternatively, a hyperplane separates $A$ and $X \setminus A$).

For example, the $O(n^{5/2})$ bound by Sharir, Smorodinsky, and Tardos on halving planes in $\rr^3$ follows from a bound of $O(nk^{3/2})$ on the number of $k$-sets in $\rr^3$.  For precise statements, we recommend Wagner's comprehensive survey on $k$-sets and their applications \cite{Wagner:2008vq}.

The number of $k$-sets of planar sets is closely related to the rectilinear crossing number of complete graphs.  Crossing numbers escape the scope of this survey, but the interested reader can learn about this connection in the survey by Ábrego, Fernández-Merchant, and Salazar \cite{Abrego:2012fl}.  The number of halving planes of a set in $\rr^3$ is also related to the planar problem of finding points covered by many triangles with vertices in a given set of points \cite{Aronov:1991hp}.

\subsection{The same-type lemma}\label{subsec:same-type}

Given two $n$-tuples $(y_1, \dots, y_n)$ and $(z_1,\dots, z_{n})$ of points in $\rr^d$ we are interested in whether they are combinatorially equivalent or not.  We say they have the \textit{same type} if each pair of simplices $(y_{i_1}, \dots, y_{i_{d+1}})$ and $(z_{i_1}, \dots, z_{i_{d+1}})$ have the same orientation.  The orientation of the simplex $(y_{i_1}, \dots, y_{i_{d+1}})$ can be defined as the sign of the determinant of the $d \times d$ matrix $Y$ where the $j$-th column is $y_{i_{j+1}}-y_{i_1}$.

A repeated application of the ham sandwich theorem was used by Bárány and Valtr to prove the following theorem \cite{Barany:1998gi}.

\begin{theorem}[Same-type lemma]
Let $n,d$ be positive integers.  There exists a constant $c = c(n,d)$ such that for any finite set $X \subset \rr^d$ we can find $n$ pairwise disjoint subsets $Y_1, \dots, Y_n$ such that
\begin{itemize}
\item for each $i$, $|Y_i| \ge c |X|$ and
\item every $n$-tuple $(y_1, \dots, y_n)$ such that $y_i \in Y_i$ for all $i=1,\dots, n$ has the same type.
\end{itemize}
\end{theorem}

The best bound for the constant $c(n,d)$ is $c(n,d) = 2^{-O(d^3 n \log n)}$ by Fox, Pach, and Suk \cite{Fox:2016bp}.  The main result of Fox, Pach, and Suk (a regularity lemma for semialgebraic hypergraphs) is much more general.  One of the key points in their proof is an application of another mass partitioning theorem by Chazelle \cite{Chazelle:1993hv} in the vein of Theorem \ref{theorem-cutting-lemma}.  A much more general result regarding partitions of a set of points in $\rr^d$ into convex clusters was proved by Pór and Valtr \cite{Por:2002ix}.

\section*{Acknowledgments}

The authors thank Jesús de Loera, Ruy Fábila, Gabriel Nivasch, Florian Frick, Peter Landweber, Luis Montejano, Dömötör Pálvölgyi, Adam Sheffer, Jorge Urrutia, and the anonymous referee for their comments.  We also thank Peter Winkler and Arseniy Akopyan for pointing out Example \ref{example-russian} and the reference to its origin.

\newcommand{\etalchar}[1]{$^{#1}$}
\providecommand{\bysame}{\leavevmode\hbox to3em{\hrulefill}\thinspace}
\providecommand{\MR}{\relax\ifhmode\unskip\space\fi MR }
\providecommand{\MRhref}[2]{%
  \href{http://www.ams.org/mathscinet-getitem?mr=#1}{#2}
}
\providecommand{\href}[2]{#2}

\end{document}